\documentclass[11pt,leqno]{article}

\usepackage{amsfonts,latexsym,amsmath,amssymb,amsthm}
\usepackage{fullpage}
\usepackage{hyperref}
\hypersetup{hidelinks}
\usepackage{setspace}

\newtheorem{theorem}{Theorem}[section]
\newtheorem{lemma}[theorem]{Lemma}
\newtheorem{proposition}[theorem]{Proposition}
\newtheorem{corollary}[theorem]{Corollary}

\numberwithin{equation}{section}

\newcommand{\ls}{\leqslant}
\newcommand{\gr}{\geqslant}
\newcommand{\E}{\mathbb{E}}

\newcommand{\R}{\mathbb{R}}

\newcommand{\conv}{{\rm conv}}

\begin{document}
\small

\title{\bf Moment comparisons, Sudakov inequalities and entropy of centroid bodies}
\author{Antonios Hmadi and Dimitrios-Marios Liakopoulos}
\date{}
\maketitle

\begin{abstract}\footnotesize
Let $X$ be an isotropic log-concave random vector in $\R^n$ and let $G$ be standard Gaussian. 
Starting from the first moment comparison for gauges, we derive corresponding estimates at arbitrary moment orders. 
For every gauge $\phi$ and every $q\gr 1$,
$$ \|\phi(G)\|_q\ls C\sqrt{\ln(en)+q}\,\|\phi(X)\|_q,\qquad \|\phi(X)\|_q\ls C\left(\sqrt{\ln(en)}+\psi(X)\sqrt q\right)\|\phi(G)\|_q. $$
Applied to support functions, this gives the sharp worst case order $C\sqrt{\ln(en)}$ for the $L_2$-Sudakov constant and quantitative $L_p$-Sudakov estimates. 
Combining, at each level of Lata\l a's dyadic chain, the strongest of the three quantitative $L_p$-Sudakov estimates used here yields
$$ \left(\E\|Y\|^p\right)^{1/p}\ls C\left(n^{1/4}\sqrt{\ln(en)}\,\ln(e+\ln(en))\,\E\|X\|+\sigma_p(Y)\right) $$
whenever the weak moments of $Y$ are dominated by those of $X$. 
In a second direction, we study the generalized dual Sudakov problem for the self generated metrics associated with $Z_r(X)^\circ$ and prove dimension free packing estimates for $Z_p(X)$. 
We also obtain mean norm estimates for centroid bodies, factorization through arbitrary symmetric convex bodies and an affine dimensional refinement.
\end{abstract}

%%%%%%%%%%%%%%%%%%%%%%%%%%%%%%%%%%%%%%%%%%%%%%%%%%%%%%%%%%%%%%%%%%%%%%%%%%%%%%%%%%%%%%%%%%%%%%%%%%%%%%%%%%%%%%%%%%%%%%%%%%%%%%%%%%%%%%
\section{Introduction}
%%%%%%%%%%%%%%%%%%%%%%%%%%%%%%%%%%%%%%%%%%%%%%%%%%%%%%%%%%%%%%%%%%%%%%%%%%%%%%%%%%%%%%%%%%%%%%%%%%%%%%%%%%%%%%%%%%%%%%%%%%%%%%%%%%%%%%

Sudakov's minoration~\cite{Sudakov-1969} bounds the expected supremum of a Gaussian process from below in terms of metric entropy.
Its extensions to log-concave vectors are linked to comparison inequalities for gauges, to $L_p$-Sudakov minoration and to the relation between weak and strong moments; see~\cite{Latala-Sudakov-1997,Latala-Sudakov-2014,Latala-Problems-2017}.
We consider three related questions.
The first is whether the comparison between Gaussian and log-concave expectations of gauges extends to arbitrary moments.
The second is what quantitative $L_2$- and $L_p$-Sudakov estimates follow from such a comparison and what they imply for weak and strong moments.
The third concerns the generalized dual Sudakov problem for the self generated metrics associated with polar centroid bodies.

For isotropic log-concave vectors, Eldan and Lehec~\cite{Eldan-2013,Eldan-Lehec-2014} proved the upper comparison of expectations of gauges with Gaussian expectations, while Bizeul and Klartag~\cite{Bizeul-Klartag-2025} proved the reverse comparison.
Together with Letwin's estimate, which makes the constants absolute~\cite{Letwin-2026}, these results give the first moment comparison recorded by Bizeul~\cite[Theorem~1.2]{Bizeul-MMstar-2026}.
Recall that a gauge is a nonnegative convex positively homogeneous function.
A random vector is unconditional if its distribution is invariant under all coordinatewise sign changes.
For an isotropic log-concave vector $X$, let $\psi(X)$ denote its Poincar\'e constant, with the normalization fixed in Section~\ref{section:preliminaries}.

\begin{theorem}\label{th:intro-moment-comparison}
Let $n\gr2$, let $X$ be isotropic and log-concave in $\R^n$, let $G$ be standard Gaussian, and let $\phi:\R^n\to[0,\infty)$ be a gauge.
\begin{enumerate}
\item[{\rm (i)}] For every $q\gr1$,
\begin{equation}\label{eq:reverse-moment-comparison}
\|\phi(G)\|_q\ls C\sqrt{\ln(en)+q}\,\|\phi(X)\|_q.
\end{equation}
\item[{\rm (ii)}] For every $q\gr1$,
\begin{equation}\label{eq:upper-moment-comparison}
\|\phi(X)\|_q \ls C\left(\sqrt{\ln(en)}+\psi(X)\sqrt q\right)\|\phi(G)\|_q.
\end{equation}
\item[{\rm (iii)}] Consequently,
$$ \frac{c}{\sqrt{\ln(en)}}\|\phi(G)\|_q \ls\|\phi(X)\|_q \ls C\sqrt{\ln(en)}\,\|\phi(G)\|_q $$
whenever
\begin{equation}\label{eq:two-sided-moment-range}
1\ls q\ls \frac{c\ln(en)}{\psi(X)^2}.
\end{equation}
In particular, the conclusion holds uniformly over $X$ for $1\ls q\ls c\sqrt{\ln(en)}$.
\end{enumerate}
\end{theorem}

The proof of Theorem~\ref{th:intro-moment-comparison} is direct once the first moment comparison is available.
Part~(i) controls the Euclidean Lipschitz constant of the gauge by $\E\phi(X)$ and then uses Gaussian concentration.
For $q\gr2$, part~(ii) follows from the Poincar\'e moment inequality applied to the seminorm $\phi(x)+\phi(-x)$, and the remaining orders follow from monotonicity.
The cube shows that the dimension and moment contributions in~\eqref{eq:reverse-moment-comparison} are both necessary.
For unconditional $X$, Theorem~\ref{th:intro-unconditional-moment-comparison} below gives, for every $q\gr1$,
$$ \frac{c}{\sqrt{\ln(en)+q}}\|\phi(G)\|_q \ls \|\phi(X)\|_q \ls C\sqrt{\ln(en)+q}\,\|\phi(G)\|_q. $$
The upper estimate uses Lata\l a's unconditional weak--strong moment theorem and the first moment comparison for the product exponential measure.
The isotropic cube and the isotropic product exponential measure show that the dimension and moment contributions have the correct order.

Theorem~\ref{th:intro-moment-comparison} should be distinguished from the mixed inequality of Bizeul~\cite[Theorem~1.3]{Bizeul-MMstar-2026}.
If $K$ is a centered convex body and $X_K$ is uniformly distributed on $K$, then $\bigl(\E h_K(G)\bigr)^2\ls2n\,\E h_K(X_K)$.
Here the same body determines both the support function and the probability measure.
Combined with the upper first moment comparison, this gives Bizeul's optimal estimate $M^*(K)\ls C\sqrt{n\ln(en)}$ in isotropic position.
In Theorem~\ref{th:intro-moment-comparison}, by contrast, the gauge and the isotropic log-concave vector are unrelated; the price in the upper comparison is the term $\psi(X)\sqrt q$.

\medskip 

Our first application concerns Sudakov minoration.
We write $N(A,B)$ for the least number of translates of $B$ needed to cover $A$.
The classical Gaussian Sudakov inequality states that, for every nonempty bounded $T\subseteq\R^n$ and every $\varepsilon>0$,
$$ \E\sup_{t\in T}\langle t,G\rangle\gr c\varepsilon\sqrt{\ln N(T,\varepsilon B_2^n)}. $$
For an isotropic log-concave vector $X$, let $C_X$ be the least constant such that
$$ \E\sup_{t\in T}\langle t,X\rangle\gr \frac{\varepsilon}{C_X}\sqrt{\ln N(T,\varepsilon B_2^n)} $$
for every such $T$ and $\varepsilon$.
Theorem~\ref{th:intro-moment-comparison}(i), applied at first moment to support functions, gives $C_X\ls C\sqrt{\ln(en)}$.
The isotropic cube gives the reverse lower bound in the worst case, and hence $\sup_X C_X\simeq\sqrt{\ln(en)},$ where the supremum runs over all isotropic log-concave vectors in $\R^n$.
This is an $L_2$-Sudakov statement because isotropy identifies the $L_2$ distance of the linear functionals with the Euclidean distance.
The $L_2$-Sudakov constant and its relation to polar centroid bodies were considered by Lata\l a in~\cite{Latala-Zp-2019}.
In the notation of that paper, $M_p(X)=Z_p(X)^\circ$, and~\cite[Proposition~14]{Latala-Zp-2019} states that
$$ N\left(Z_p(X)^\circ,\frac{eC_X}{\sqrt p}B_2^n\right)\ls e^p,\qquad p\gr2. $$
Thus the mechanism connecting $C_X$ with the entropy of $Z_p(X)^\circ$ is already known, while the estimate above identifies the sharp worst case order of $C_X$ for isotropic log-concave vectors.

\medskip 

We next turn to Lata\l a's $L_p$ formulation.
For a compact convex set $K\subseteq\R^n$, let $h_K(x)=\sup_{y\in K}\langle x,y\rangle$ be its support function.
If $K$ contains the origin in its interior, put $\|x\|_K=\inf\{t>0:x\in tK\}$ and $K^\circ=\{x:h_K(x)\ls1\}$.
We write $\conv A$ for the convex hull of a set $A$.
For $p\gr1$ and a nondegenerate random vector $X$ with finite $p$-th moments, let $Z_p(X)$ be the origin symmetric convex body whose support function is
$$ h_{Z_p(X)}(\theta)=\left(\E|\langle X,\theta\rangle|^p\right)^{1/p}. $$
For a random vector $X$, define $d_{X,p}(s,t)=\|\langle s-t,X\rangle\|_p$.
Following Lata\l a~\cite{Latala-Sudakov-2014}, we say that $X$ satisfies $\mathrm{SMP}_p(\alpha)$ if, whenever $T\subseteq\R^n$ is finite, $|T|\gr e^p$, and $d_{X,p}(s,t)\gr A$ for distinct $s,t\in T$, one has
$$ \E\sup_{s,t\in T}\langle t-s,X\rangle\gr\alpha A. $$
We write $\mathrm{SMP}(\alpha)$ if this holds for every $p\gr1$.
Lata\l a conjectured that every finite dimensional log-concave random vector satisfies $\mathrm{SMP}(c)$ with an absolute constant.
In general he proved $\mathrm{SMP}_p(c/\max\{p,2\})$ and, in dimension $n$, $\mathrm{SMP}(c/n)$; for symmetric vectors he also obtained an absolute constant once $p\gr2n\ln(n+e)$.

\begin{theorem}\label{th:intro-SMP-summary}
Let $X$ be isotropic and log-concave in $\R^n$, where $n\gr2$.
For every $2\ls p\ls n$,
\begin{equation}\label{eq:intro-SMP-summary}
X\text{ satisfies }\mathrm{SMP}_p\left(c\max\left\{\frac1p,\frac1{\sqrt{p\ln(en)}},\frac{\sqrt p}{\sqrt{n\ln(en)\ln(e+p)}}\right\}\right).
\end{equation}
Moreover, every nondegenerate $n$-dimensional log-concave random vector satisfies
\begin{equation}\label{eq:intro-global-SMP-summary}
\mathrm{SMP}\left(\frac{c}{\sqrt{n\ln(en)}}\right).
\end{equation}
\end{theorem}

The first term in~\eqref{eq:intro-SMP-summary} is the general estimate of Lata\l a.
The second term follows from the Gaussian Sudakov minoration and Theorem~\ref{th:intro-moment-comparison}; it improves $1/p$ when $p\gr C\ln(en)$.
The third term uses the mean width estimate for $Z_p(X)$ of Giannopoulos, Pafis and Tziotziou~\cite{Giannopoulos-Pafis-Tziotziou-2026} and improves the second term when $p^2\gr Cn\ln(e+p)$.
The global estimate~\eqref{eq:intro-global-SMP-summary} improves Lata\l a's general dimensional bound $c/n$ to $c/\sqrt{n\ln(en)}$.
For very large $p$, Lata\l a's exponential in $p/n$ estimate is stronger.
Thus Theorem~\ref{th:intro-SMP-summary} gives quantitative improvements in intermediate ranges but does not resolve the conjectural dimension free $\mathrm{SMP}(c)$.

\medskip 

The scale dependence in Theorem~\ref{th:intro-SMP-summary} is useful for the weak--strong problem.
For a norm $\|\cdot\|$, let $\|\cdot\|_*$ denote its dual norm and, for a random vector $Y$ and $p\gr1$, put $\sigma_p(Y)=\sup_{\|t\|_*\ls1}\|\langle t,Y\rangle\|_p$.
Lata\l a conjectured~\cite{Latala-weak-strong-2011} that, for every log-concave random vector $X$ in $\R^n$, every norm on $\R^n$ and every $p\gr1$,
$$ \left(\E\|X\|^p\right)^{1/p}\ls C\left(\E\|X\|+\sigma_p(X)\right) $$
with an absolute constant $C$.
To the best of our knowledge, the conjecture is known in several special situations, but remains open for arbitrary log-concave vectors and arbitrary norms.
A general dimension dependent estimate follows from the connection with Sudakov minoration established by Lata\l a~\cite[Corollary~6.4]{Latala-Sudakov-2014}.
More precisely, if $X$ satisfies $\mathrm{SMP}(\kappa)$ and the weak moments of $Y$ are dominated by those of $X$, then
$$ \left(\E\|Y\|^p\right)^{1/p}\ls C\left(\frac1\kappa\max\left\{1,\ln\left(\frac{en}{p}\right)\right\}\E\|X\|+\sigma_p(Y)\right). $$
Since every $n$-dimensional log-concave random vector was known to satisfy $\mathrm{SMP}(c/n)$, this gives in general a coefficient of order
$$ n\max\left\{1,\ln\left(\frac{en}{p}\right)\right\} $$
in front of the first moment.
Using only the global estimate $\mathrm{SMP}(c/\sqrt{n\ln(en)})$ from Theorem~\ref{th:intro-SMP-summary} in the same argument would reduce this coefficient to order
$$ \sqrt{n\ln(en)}\max\left\{1,\ln\left(\frac{en}{p}\right)\right\}. $$
Our next result instead retains, at each successive dyadic level, the strongest among the three estimates in~\eqref{eq:intro-SMP-summary} and gives a further improvement.

\begin{theorem}\label{th:intro-weak-strong}
Let $X$ be isotropic and log-concave in $\R^n$, and let $Y$ be a random vector satisfying
\begin{equation}\label{eq:weak-domination}
\|\langle t,Y\rangle\|_q\ls\|\langle t,X\rangle\|_q \qquad(t\in\R^n,\ q\gr1).
\end{equation}
Then, for every norm $\|\cdot\|$ on $\R^n$ and every $p\gr2$,
\begin{equation}\label{eq:weak-strong-scale-dependent}
\left(\E\|Y\|^p\right)^{1/p} \ls C\left(n^{1/4}\sqrt{\ln(en)}\,\ln(e+\ln(en))\,\E\|X\|+\sigma_p(Y)\right).
\end{equation}
In particular,
\begin{equation}\label{eq:weak-strong-X}
\left(\E\|X\|^p\right)^{1/p} \ls C\left(n^{1/4}\sqrt{\ln(en)}\,\ln(e+\ln(en))\,\E\|X\|+\sigma_p(X)\right).
\end{equation}
\end{theorem}

Using a single global Sudakov constant instead gives the weaker coefficient $C\sqrt{n\ln(en)}$.
The proof of Theorem~\ref{th:intro-weak-strong} uses Paouris' estimate at the lower levels, Borell's moment comparison in the intermediate range and the full range mean width estimate of Giannopoulos, Pafis and Tziotziou at the upper levels.
The factor $\ln(e+\ln(en))$ records the number of intermediate dyadic levels and is not claimed to be optimal.

The reverse moment comparison also gives useful information on the dual mean parameter of centroid bodies.
For an origin symmetric convex body $K$, write $M(K)=\int_{S^{n-1}}\|\theta\|_K\,d\sigma(\theta)$ and $M^*(K)=\int_{S^{n-1}}h_K(\theta)\,d\sigma(\theta)$.
Mean norms of centroid bodies were studied earlier in~\cite{Giannopoulos-Stavrakakis-Tsolomitis-Vritsiou-2015,Giannopoulos-Milman-2014}; their estimates are stronger for some small values of $p$.
Here we obtain the following estimate throughout the full range:
$$ M(Z_p(X))\ls C\sqrt{\frac{(n+p)\ln(en)}{np}},\qquad p\gr2, $$
and, together with the known estimates for $M^*(Z_p(X))$, corresponding $MM^*$ bounds.
The earlier estimates and the bound above are complementary.
These results are placed after the weak--strong theorem because they form the transition from support functions and mean widths to polar metrics.

\medskip 

Our next results concern generalized dual Sudakov estimates.
For a random vector $X$ and an origin symmetric convex body $K$, put $I_1(X,K)=\E\|X\|_K$ and $I_1^*(X,K)=\E h_K(X)$ whenever these expectations are finite.
We use the shorthand $\mathcal I_r(X)=I_1(X,Z_r(X)^\circ)=\E h_{Z_r(X)}(X)$.
For a bounded set $A$ and an origin symmetric convex body $B$, let $\mathsf M(A,B)$ be the largest cardinality of a subset $S\subseteq A$ such that $s-t\notin B$ for distinct $s,t\in S$.
Mendelson, Milman and Paouris~\cite{Mendelson-Milman-Paouris} proposed the dimension free estimate
$$ \mathsf M\left(Z_p(\mu),C I_1(\mu,K)K\right)\ls e^{Cp} $$
for every origin symmetric convex body $K$ and every origin symmetric log-concave probability measure $\mu$.
Their program gives a general dimension dependent weak estimate and seeks a separation preserving dimension reduction; they obtain the required reduction for ellipsoids and, up to logarithmic losses, for cubes.
We consider the particular choice $ K=Z_r(X)^\circ.$
Here we use the term self generated to emphasize that the body defining the metric is constructed from the same random vector $X$ as the centroid body which is being packed.
Indeed,
$$ \|s-t\|_{Z_r(X)^\circ}=h_{Z_r(X)}(s-t)=\|\langle s-t,X\rangle\|_r=d_{X,r}(s,t), $$
while
$$ I_1(X,Z_r(X)^\circ)=\E\|X\|_{Z_r(X)^\circ}=\mathcal I_r(X). $$
Thus $\mathsf M(Z_p(X),a\mathcal I_r(X)Z_r(X)^\circ)$ measures the maximal size of a subset of $Z_p(X)$ which is separated in the intrinsic metric $d_{X,r}$ at a fixed multiple of its mean scale.
In this sense the self generated case is the generalized dual Sudakov problem associated with the same moment metrics that occur in the $L_p$-Sudakov estimates above.
It is also sufficiently structured to be compared with the covariance ellipsoid: the inclusion $Z_r(X)\subseteq CrZ_2(X)$ and the lower estimate $\mathcal I_r(X)\gr c\sqrt{nr}$ will allow us to obtain a dimension free bound.

\begin{theorem}\label{th:polar-centroid-entropy}
There is an absolute constant $C\gr1$ such that, for every centered nondegenerate log-concave random vector $X$ in $\R^n$, every $r\gr2$, every $p\gr1$ and every $u\gr1$,
\begin{equation}\label{eq:polar-centroid-entropy}
\ln\mathsf M\left(Z_p(X),Cu\mathcal I_r(X)Z_r(X)^\circ\right) \ls Cp\left[ \left(\frac{r\mathcal I_2(X)}{u\mathcal I_r(X)}\right)^2+ \frac{r\mathcal I_2(X)}{u\mathcal I_r(X)} \right].
\end{equation}
Consequently,
\begin{equation}\label{eq:polar-centroid-entropy-universal}
\ln\mathsf M\left(Z_p(X),Cu\mathcal I_r(X)Z_r(X)^\circ\right) \ls Cp\left[\left(\frac ru\right)^2+\frac ru\right].
\end{equation}
If, in addition, $X$ is isotropic and $2\ls r\ls n$, then
\begin{equation}\label{eq:polar-centroid-entropy-isotropic}
\ln\mathsf M\left(Z_p(X),Cu\mathcal I_r(X)Z_r(X)^\circ\right) \ls Cp\left(\frac r{u^2}+\frac{\sqrt r}{u}\right).
\end{equation}
In particular, taking $u=\sqrt r$ gives
$$\mathsf M\left(Z_p(X),C\sqrt r\,\mathcal I_r(X)Z_r(X)^\circ\right)\ls e^{Cp}, \qquad 2\ls r\ls n.$$
\end{theorem}

The Mendelson--Milman--Paouris conjecture, specialized to the self generated choice $K=Z_r(X)^\circ$, predicts
$$ \mathsf M\left(Z_p(X),C\mathcal I_r(X)Z_r(X)^\circ\right)\ls e^{Cp}. $$
Thus the conjectural $e^{Cp}$ estimate corresponds precisely to the scale $u=1$ in Theorem~\ref{th:polar-centroid-entropy}.
Our estimate is dimension free in isotropic position, but at this scale it gives only
$$ \ln\mathsf M\left(Z_p(X),C\mathcal I_r(X)Z_r(X)^\circ\right)\ls Cpr. $$
The $e^{Cp}$ packing estimate is obtained after enlarging the separation body by the factor $\sqrt r$, namely
$$ \mathsf M\left(Z_p(X),C\sqrt r\,\mathcal I_r(X)Z_r(X)^\circ\right)\ls e^{Cp}. $$
Consequently, Theorem~\ref{th:polar-centroid-entropy} does not settle the Mendelson--Milman--Paouris conjecture even for the self generated target.
Its contribution is a dimension free estimate for this intrinsic family of metrics, with an explicit description of the remaining loss in the moment parameter $r$.
The detailed comparison with the dimension dependent Mendelson--Milman--Paouris bounds is recorded in Appendix~\ref{appendix:MMP-comparison}.

The entropy theorem yields factorization estimates through arbitrary symmetric convex bodies.
Combining its localization with the projection geometry of the subgaussian body gives a refinement in terms of the affine dimension of a separated set.
A simple consequence is the following: if $T\subseteq Z_p(X)$ is $a\mathcal I_r(X)Z_r(X)^\circ$-separated, $d=\dim\operatorname{aff}T$ and $d\ln(e+C\sqrt r/a)\ls cp$, then $|T|\ls e^{Cp}$.
The full affine dimensional estimate, which contains an additional regime involving projections, is stated in Section~\ref{sec:affine-applications}.
These are consequences of the entropy mechanism rather than additional main problems.

For context, Paouris and Pathak~\cite[Theorem~3.4]{Paouris-Pathak-2026}
recently proved a refined first moment comparison for gauges of conditioned
Gaussian measures in Bobkov position, a more restrictive class of measures
than in Theorem~\ref{th:intro-moment-comparison}.
They also obtained optimal affine mean width and Euclidean metric entropy
estimates for arbitrary convex bodies; these optimize over a linear position
and use Euclidean covering bodies, whereas the entropy estimates here are
formulated in fixed isotropic coordinates and for prescribed target bodies.

Section~\ref{section:preliminaries} fixes the notation and records the external estimates.
Section~\ref{sec:L2-proof} proves Theorem~\ref{th:intro-moment-comparison}, its unconditional extension and the sharp worst case $L_2$-Sudakov estimate.
Section~\ref{sec:Lp-SMP} proves the $L_p$-Sudakov estimates summarized in Theorem~\ref{th:intro-SMP-summary}.
Section~\ref{sec:weak-strong} proves Theorem~\ref{th:intro-weak-strong} and then records the mean norm consequences for centroid bodies.
Section~\ref{sec:generalized} proves Theorem~\ref{th:polar-centroid-entropy} and the factorization results.
Section~\ref{sec:affine-applications} gives the affine dimensional refinement and the convex body application.
Additional quantile estimates, a dimension dependent large family bound and the detailed comparison with the Mendelson--Milman--Paouris estimates are collected in the appendices.

%%%%%%%%%%%%%%%%%%%%%%%%%%%%%%%%%%%%%%%%%%%%%%%%%%%%%%%%%%%%%%%%%%%%%%%%%%%%%%%%%%%%%%%%%%%%%%%%%%%%%%%%%%%%%%%%%%%%%%%%%%%%%%%%%%%%%%
\section{Notation and preliminary results}\label{section:preliminaries}
%%%%%%%%%%%%%%%%%%%%%%%%%%%%%%%%%%%%%%%%%%%%%%%%%%%%%%%%%%%%%%%%%%%%%%%%%%%%%%%%%%%%%%%%%%%%%%%%%%%%%%%%%%%%%%%%%%%%%%%%%%%%%%%%%%%%%%

We use the Euclidean structure of $\R^n$: $\langle\cdot,\cdot\rangle$, $|\cdot|$, $B_2^n$ and $(e_i)_{i=1}^n$ denote the Euclidean inner product, norm, unit ball and canonical basis.
For background on the notation and standard facts from asymptotic geometric analysis used below, we refer to~\cite{Artstein-Avidan-Giannopoulos-Milman-I,Artstein-Avidan-Giannopoulos-Milman-II}.
According to context, $|\cdot|$ also denotes cardinality or Lebesgue measure.
We write $\mathbb P$ and $\E$ for probability and expectation.
For a real random variable $\xi$ and $p>0$, put $\|\xi\|_p=(\E|\xi|^p)^{1/p}$.
The relation $a\simeq b$ means $ca\ls b\ls Ca$ for absolute constants, and $c,C,c_1,C_1,\ldots$ denote positive absolute constants which may change from line to line.

A random vector is centered if $\E X=0$.
A probability measure $\mu$ on $\R^n$ is log-concave if, for all nonempty compact sets $A,B\subseteq\R^n$ and every $0<\lambda<1$,
$$ \mu((1-\lambda)A+\lambda B)\gr\mu(A)^{1-\lambda}\mu(B)^\lambda. $$
A random vector is log-concave if its distribution is log-concave; in the nondegenerate case this is equivalent to having a density of the form $e^{-V}$ with $V:\R^n\to(-\infty,\infty]$ convex.
It is symmetric if $X$ and $-X$ have the same distribution.
Its covariance matrix is $ \operatorname{Cov}(X)=\E[(X-\E X)\otimes(X-\E X)], $ where $(x\otimes y)z=\langle y,z\rangle x$.
The vector $X$ is isotropic if it is centered and $\operatorname{Cov}(X)=I_n$, where $I_n$ is the identity operator, and is nondegenerate if $\operatorname{Cov}(X)$ is invertible.
Linear images preserve log-concavity; in particular, orthogonal marginals of an isotropic log-concave vector are isotropic and log-concave on the corresponding subspace.

A convex body is a compact convex set with nonempty interior; it is origin symmetric if $K=-K$.
We write $\operatorname{span}A$ and $\operatorname{aff}A$ for the linear and affine hulls of a set $A$.
If $K$ is origin symmetric, then the Minkowski functional $\|\cdot\|_K$ defined above is a norm.
For a bounded set $A\subseteq\R^n$ and an origin symmetric convex body $B$, let $N(A,B)$ be the least positive integer $m$ for which
$$ A\subseteq\bigcup_{i=1}^m(x_i+B) $$
for some $x_1,\ldots,x_m\in\R^n$.
Together with the packing number $\mathsf M(A,B)$ defined in the Introduction, we use without further comment
$$ \mathsf M(A,2B)\ls N(A,B)\ls\mathsf M(A,B), \qquad N(A,D)\ls N(A,B)N(B,D), $$
whenever the sets involved make these quantities finite.
For an origin symmetric convex body $K$ put
$$ \ell(K)=\E\|G\|_K,\qquad \ell^*(K)=\E h_K(G), $$
where $G$ is a standard Gaussian vector in $\R^n$.
Then
\begin{equation}\label{eq:ell-M-relations}
\ell(K)\simeq\sqrt n\,M(K),\qquad \ell^*(K)\simeq\sqrt n\,M^*(K).
\end{equation}
For a subspace $E\subseteq\R^n$, we denote by $P_E$ the orthogonal projection onto $E$, by $B_2^E=B_2^n\cap E$ its Euclidean unit ball and by $G_E$ a standard Gaussian vector in $E$.
If $A\subseteq E$ is measurable and has positive volume, we write
$$ \operatorname{vrad}_E(A)=\left(\frac{|A|}{|B_2^E|}\right)^{1/\dim E}. $$
The Gaussian Sudakov minoration~\cite{Sudakov-1969} will be used in the following form: for every nonempty bounded $T\subseteq\R^n$ and every $\varepsilon>0$,
\begin{equation}\label{eq:Gaussian-Sudakov}
\E\sup_{t\in T}\langle t,G\rangle\gr c\varepsilon\sqrt{\ln N(T,\varepsilon B_2^n)}.
\end{equation}
We shall also use the classical Sudakov inequality and the dual Sudakov inequality of Pajor and Tomczak-Jaegermann~\cite{Sudakov-1969,Pajor-Tomczak-1986}: for every origin symmetric convex body $K$ and every $t>0$,
\begin{equation}\label{eq:classical-Sudakov-covering}
\ln N(K,tB_2^n)\ls\frac{C\ell^*(K)^2}{t^2}, \qquad \ln N(B_2^n,tK)\ls\frac{C\ell(K)^2}{t^2}.
\end{equation}
We use the square root convention for the Poincar\'e constant: for an isotropic log-concave $X$, let $\psi(X)$ be the least constant such that $\operatorname{Var}f(X)\ls\psi(X)^2\E|\nabla f(X)|^2$ for every locally Lipschitz function $f$.
We write $\psi_n$ for the supremum of $\psi(X)$ over all isotropic log-concave random vectors in $\R^n$.
The first moment comparison used below is the form recorded by Bizeul~\cite[Theorem~1.2]{Bizeul-MMstar-2026}.
The reverse inequality is due to Bizeul and Klartag~\cite{Bizeul-Klartag-2025}, and the upper inequality to Eldan and Lehec~\cite{Eldan-2013,Eldan-Lehec-2014}.
Letwin's estimate makes the constants absolute~\cite{Letwin-2026}.
Thus, for every gauge $\phi$,
\begin{equation}\label{eq:gauge-comparison}
 \frac{c}{\sqrt{\ln(en)}}\E\phi(G)\ls\E\phi(X)\ls C\sqrt{\ln(en)}\E\phi(G).
\end{equation}
Letwin also proved $\psi_n\ls C\ln^{1/4}(en)$.

We shall use two estimates for mean widths of centroid bodies.
Paouris proved~\cite{Paouris-2006}
\begin{equation}\label{eq:Paouris-small-p-Mstar}
M^*(Z_q(X))\ls C\sqrt q,
\qquad 1\ls q\ls\sqrt n.
\end{equation}
Earlier full range estimates were obtained by E.~Milman~\cite{EMilman-mean-width-2015}.
In the notation of Giannopoulos, Pafis and Tziotziou, the mean width $w$ is our parameter $M^*$.
Using Bizeul's optimal mean width estimate and a projection argument, they proved the following bound~\cite[Equation~(3.5)]{Giannopoulos-Pafis-Tziotziou-2026}.
\begin{lemma}\label{lem:full-range-centroid-width}
Let $X$ be isotropic and log-concave in $\R^n$.
Then, for every $1\ls q\ls n$,
$$ M^*(Z_q(X))\ls C\sqrt{q\ln(e+q)}, \qquad \ell^*(Z_q(X))\ls C\sqrt{nq\ln(e+q)}. $$
\end{lemma}
The second estimate follows from~\eqref{eq:ell-M-relations}.
For $q\ls\sqrt n$, \eqref{eq:Paouris-small-p-Mstar} gives the sharper bound $\ell^*(Z_q(X))\ls C\sqrt{nq}$.

We use the pseudometric $d_{X,p}$ and the properties $\mathrm{SMP}_p(\alpha)$ and $\mathrm{SMP}(\alpha)$ as defined in the Introduction.
We shall use the following consequences of Lata\l a's results~\cite[Lemmas~2.1, 2.2, 2.6, 2.8 and Remarks~2.4, 2.7]{Latala-Sudakov-2014}.

\begin{lemma}\label{lem:Latala-tools}
\begin{enumerate}
\item[{\rm (1)}]  Let $X$ be a random vector, let $X'$ be an independent copy of $X$, and let $p\gr1$.
If $X-X'$ satisfies $\mathrm{SMP}_p(\alpha)$, then $X$ satisfies $\mathrm{SMP}_p(\min\{1/2,\alpha/4\})$.
The property $\mathrm{SMP}_p(\alpha)$ is preserved under linear images.
\item[{\rm (2)}]  If $\xi$ is centered and log-concave on $\R$, then, for every $p\gr2$,
\begin{equation}\label{eq:one-dimensional-moment-growth}
\|\xi\|_p\ls Cp\|\xi\|_2.
\end{equation}
\item[{\rm (3)}]  If $Y$ is symmetric, log-concave and $d$-dimensional and satisfies $\mathrm{SMP}_d(\alpha)$, then it satisfies $\mathrm{SMP}_p(\alpha/8)$ for every $p\gr d$.
\item[{\rm (4)}]  Every log-concave random vector satisfies $\mathrm{SMP}_p(c/\max\{p,2\})$.
Every $d$-dimensional log-concave random vector satisfies $\mathrm{SMP}(c/d)$.
\item[{\rm (5)}]  If $Y$ is symmetric, log-concave and $d$-dimensional, then, for every $p\gr2$,
\begin{equation}\label{eq:Latala-large-p}
Y\text{ satisfies }\mathrm{SMP}_p\left(\frac{e^{p/d}-1}{\sqrt2\,p}\right).
\end{equation}
\end{enumerate}
\end{lemma}

The one-dimensional estimate in Lemma~\ref{lem:Latala-tools}(2) follows by symmetrization from~\cite[Lemma~2.3]{Latala-Sudakov-2014}; the remaining assertions are the corresponding forms of the results cited above.

We shall use the following consequence of Lata\l a and Nayar~\cite[Theorem~1]{Latala-Nayar-2020}.

\begin{lemma}
For every nondegenerate $n$-dimensional random vector $Y$ and every $p\gr2$,
\begin{equation}\label{eq:LN-self-gauge}
\left(\E\|Y\|_{Z_p(Y)}^p\right)^{1/p} \ls2\sqrt e\sqrt{\frac{n+p}{p}}.
\end{equation}
\end{lemma}

%%%%%%%%%%%%%%%%%%%%%%%%%%%%%%%%%%%%%%%%%%%%%%%%%%%%%%%%%%%%%%%%%%%%%%%%%%%%%%%%%%%%%%%%%%%%%%%%%%%%%%%%%%%%%%%%%%%%%%%%%%%%%%%%%%%%%%
\section{Moment comparison and \texorpdfstring{$L_2$}{L2}-Sudakov minoration}\label{sec:L2-proof}
%%%%%%%%%%%%%%%%%%%%%%%%%%%%%%%%%%%%%%%%%%%%%%%%%%%%%%%%%%%%%%%%%%%%%%%%%%%%%%%%%%%%%%%%%%%%%%%%%%%%%%%%%%%%%%%%%%%%%%%%%%%%%%%%%%%%%%

We begin with the proof of Theorem~\ref{th:intro-moment-comparison}.
We then record the unconditional extension and derive the sharp worst case $L_2$-Sudakov estimate stated in the Introduction.

\begin{proof}[Proof of Theorem~$\ref{th:intro-moment-comparison}{\rm (i)}$]
Put $L=\ln(en)$.
Since $\phi$ is finite and sublinear, there is a compact convex set
$$ K_\phi=\{y\in\R^n:\langle x,y\rangle\ls\phi(x)\text{ for every }x\in\R^n\} $$
such that $\phi=h_{K_\phi}$.
Set $b=\sup_{|x|\ls1}\phi(x)=\max_{y\in K_\phi}|y|$.
If $b=0$, then $\phi=0$ and there is nothing to prove.
Choose $y_0\in K_\phi$ with $|y_0|=b$.
Since $\phi$ is nonnegative and $\phi(x)\gr\langle x,y_0\rangle$, we have $\phi(x)\gr\bigl(\langle x,y_0\rangle\bigr)_+$.

The random variable $\xi=\langle X,y_0\rangle$ is centered and log-concave, and isotropy gives $\|\xi\|_2=b$.
Borell's regular moment comparison~\cite{Borell-1974} and centeredness therefore yield $b\ls C_1\|\xi\|_1=2C_1\E\xi_+\ls 2C_1\E\phi(X)$.
Moreover, $K_\phi\subseteq bB_2^n$, so the support function $\phi$ is $b$-Lipschitz.
The Gaussian concentration inequality gives, for every $q\gr1$,
$$ \|\phi(G)\|_q\ls\E\phi(G)+C_2\sqrt q\,b. $$
The reverse first moment comparison in~\eqref{eq:gauge-comparison}, together with the preceding estimate and the Gaussian moment estimate, now gives
$$ \|\phi(G)\|_q\ls C_3(\sqrt L+\sqrt q)\E\phi(X) \ls C\sqrt{L+q}\,\|\phi(X)\|_q, $$
which proves~\eqref{eq:reverse-moment-comparison}.
\end{proof}

The cube example below shows that both terms in the factor in~\eqref{eq:reverse-moment-comparison} are necessary.
Indeed, let $X$ be uniform on the isotropic cube $[-\sqrt3,\sqrt3]^n$ and take $\phi(x)=\|x\|_\infty$.
Then $\|\phi(X)\|_q\ls\sqrt3$, whereas
$$ \|\phi(G)\|_q\gr\max\left\{\E\|G\|_\infty,\|\langle G,e_1\rangle\|_q\right\} \gr c_1\max\{\sqrt{\ln(en)},\sqrt q\}\gr c\sqrt{\ln(en)+q}. $$

\begin{proof}[Proof of Theorem~$\ref{th:intro-moment-comparison}{\rm (ii)}$]
Put $L=\ln(en)$ and define $N(x)=\phi(x)+\phi(-x)$.
Then $N$ is a seminorm and $\phi\ls N$.
Let $ b_N=\sup_{|x|\ls1}N(x) $ be its Euclidean Lipschitz constant.
A standard moment consequence of the Poincar\'e inequality, see for example~\cite{Bobkov-Ledoux-1997}, gives
$$ \|N(X)\|_q\ls \E N(X)+C\psi(X)q b_N, \qquad q\gr2. $$
The upper first moment comparison in~\eqref{eq:gauge-comparison} yields
$$ \E N(X)\ls C\sqrt L\,\E N(G) \ls C\sqrt L\,\|N(G)\|_q. $$

It remains to compare $b_N$ with a Gaussian moment.
Since $N$ is a seminorm, there is a symmetric closed convex set $B\subseteq\R^n$ such that $N(x)=\sup_{y\in B}\langle x,y\rangle$ and $b_N=\sup_{y\in B}|y|$.
Choosing $y\in B$ with $|y|$ arbitrarily close to $b_N$, we have $N(G)\gr|\langle G,y\rangle|$.
Hence $\|N(G)\|_q\gr c\sqrt q\,b_N$.
Combining the preceding estimates,
$$ \|\phi(X)\|_q \ls\|N(X)\|_q \ls C\left(\sqrt L+\psi(X)\sqrt q\right)\|N(G)\|_q. $$
Finally, Gaussian symmetry and the triangle inequality give $\|N(G)\|_q\ls\|\phi(G)\|_q+\|\phi(-G)\|_q =2\|\phi(G)\|_q$, which proves~\eqref{eq:upper-moment-comparison} for $q\gr2$.

Let now $1\ls q\ls2$.
Monotonicity and the case $q=2$ give
$$ \|\phi(X)\|_q\ls C(\sqrt L+\psi(X))\|\phi(G)\|_2. $$
Write $\phi=h_K$ for a compact convex set $K$, let $b=\max_{y\in K}|y|$, and choose $y_0\in K$ with $|y_0|=b$.
Since $\phi(G)\gr(\langle G,y_0\rangle)_+$, we have $b\ls C\E\phi(G)$.
Gaussian concentration and monotonicity therefore give $\|\phi(G)\|_2\ls C\E\phi(G)\ls C\|\phi(G)\|_q$.
Since $q\gr1$, this proves~\eqref{eq:upper-moment-comparison} and completes the proof.
\end{proof}

\begin{proof}[Proof of Theorem~$\ref{th:intro-moment-comparison}{\rm (iii)}$]
Put $L=\ln(en)$.
Applying the Poincar\'e inequality to the linear functions $x\mapsto\langle x,\theta\rangle$ and using isotropy shows that $\psi(X)\gr1$.
Thus~\eqref{eq:two-sided-moment-range} implies $q\ls cL$ and $\psi(X)\sqrt q\ls\sqrt{cL}$.
The lower estimate follows from Theorem~\ref{th:intro-moment-comparison}(i), and the upper estimate follows from part~(ii).
Finally, Letwin's bound $\psi(X)\ls\psi_n\ls C L^{1/4}$ shows that the range contains $1\ls q\ls c\sqrt L$ after changing the absolute constant.
\end{proof}

\begin{theorem}\label{th:intro-unconditional-moment-comparison}
Under the assumptions of Theorem~$\ref{th:intro-moment-comparison}$, suppose in addition that the distribution of $X$ is unconditional.
Then, for every $q\gr1$,
$$ \frac{c}{\sqrt{\ln(en)+q}}\|\phi(G)\|_q \ls \|\phi(X)\|_q \ls C\sqrt{\ln(en)+q}\,\|\phi(G)\|_q. $$
\end{theorem}

\begin{proof}
The lower estimate is Theorem~\ref{th:intro-moment-comparison}(i), so it remains to prove the upper estimate.
Let $\mathcal E=(\mathcal E_1,\ldots,\mathcal E_n)$ have independent symmetric exponential coordinates of variance one and define $N(x)=\phi(x)+\phi(-x)$.
Then $N$ is a seminorm, $\phi\ls N$, and Gaussian symmetry gives $\|N(G)\|_q\ls2\|\phi(G)\|_q$.
Replacing $N$ by $N+\varepsilon|\cdot|$ and then letting $\varepsilon\downarrow0$, it is enough to treat the case where $N$ is a norm.
We write $N^*$ for its dual norm.

Lata\l a's unconditional weak--strong moment theorem~\cite[Theorem~3.1]{Latala-weak-strong-2011} yields, for every $q\gr1$,
\begin{equation}\label{eq:Latala-unconditional-weak-strong}
\|N(X)\|_q\ls C\left(\E N(\mathcal E)
+\sup_{N^*(y)\ls1}\|\langle y,X\rangle\|_q\right).
\end{equation}
The vector $\mathcal E$ is isotropic and log-concave, so the upper first moment comparison~\eqref{eq:gauge-comparison} gives, with $L=\ln(en)$,
$$ \E N(\mathcal E)\ls C\sqrt L\,\E N(G) \ls C\sqrt L\,\|N(G)\|_q. $$

Put $b_N=\sup_{N^*(y)\ls1}|y|$.
The one-dimensional moment estimate~\eqref{eq:one-dimensional-moment-growth}, together with isotropy and monotonicity of moments when $1\ls q\ls2$, implies
$$ \sup_{N^*(y)\ls1}\|\langle y,X\rangle\|_q\ls Cq b_N, \qquad q\gr1. $$
On the other hand, since $N(x)=\sup_{N^*(y)\ls1}\langle x,y\rangle$, the Gaussian moment of a linear functional gives $\|N(G)\|_q\gr c\sqrt q\,b_N, \; q\gr1$.
Consequently, the second term in~\eqref{eq:Latala-unconditional-weak-strong} is at most $C\sqrt q\,\|N(G)\|_q$.
Combining these estimates, we obtain
$$ \|\phi(X)\|_q\ls\|N(X)\|_q \ls C_1(\sqrt L+\sqrt q)\|N(G)\|_q \ls C\sqrt{L+q}\,\|\phi(G)\|_q, $$
which completes the proof.
\end{proof}

The product exponential example below shows that the upper factor also has the correct order.
Indeed, for the isotropic product exponential vector $\mathcal E$, we have $ \E\|\mathcal E\|_\infty\simeq\ln(en)$ and $\E\|G\|_\infty\simeq\sqrt{\ln(en)}, $ while $ \|\mathcal E_1\|_q\simeq q$ and $\|G_1\|_q\simeq\sqrt q$ for $q\gr2. $

Recall from the Introduction that $C_X$ denotes the $L_2$-Sudakov constant of an isotropic log-concave vector $X$.

\begin{theorem}\label{th:intro-Sudakov}
For every isotropic log-concave random vector $X$ in $\R^n$,
$$ C_X\ls C\sqrt{\ln(en)}. $$
\end{theorem}

\begin{proof}
Fix a nonempty bounded set $T\subseteq\R^n$ and $\varepsilon>0$.
Choose $t_0\in T$, put $T_0=T-t_0$ and define $\phi(x)=\sup_{t\in T_0}\langle t,x\rangle$.
Since $0\in T_0$, $\phi$ is a gauge; since $X$ and $G$ are centered, translation of $T$ does not change either expected supremum.
The reverse first moment comparison~\eqref{eq:gauge-comparison} and~\eqref{eq:Gaussian-Sudakov} give
$$ \E\sup_{t\in T}\langle t,X\rangle =\E\phi(X) \gr\frac{c}{\sqrt{\ln(en)}}\E\phi(G) \gr\frac{c\varepsilon}{\sqrt{\ln(en)}}\sqrt{\ln N(T,\varepsilon B_2^n)}. $$
\end{proof}

The logarithmic loss is necessary.
Indeed, if $X$ has independent coordinates uniform on $[-\sqrt3,\sqrt3]$ and $T=\{e_1,\ldots,e_n\}$, then $X$ is isotropic and log-concave and
$$ N\left(T,\frac12B_2^n\right)=n, \qquad \E\sup_{t\in T}\langle t,X\rangle\ls\sqrt3. $$
Hence $C_X\gr c\sqrt{\ln(en)}$.
Together with Theorem~\ref{th:intro-Sudakov}, this shows that the worst possible $L_2$-Sudakov constant is of order $\sqrt{\ln(en)}$.

The moment and quantile refinements of the preceding argument are recorded in Appendix~\ref{appendix:additional-sudakov}.

%%%%%%%%%%%%%%%%%%%%%%%%%%%%%%%%%%%%%%%%%%%%%%%%%%%%%%%%%%%%%%%%%%%%%%%%%%%%%%%%%%%%%%%%%%%%%%%%%%%%%%%%%%%%%%%%%%%%%%%%%%%%%%%%%%%%%%
\section{\texorpdfstring{$L_p$}{Lp}-Sudakov minoration}\label{sec:Lp-SMP}
%%%%%%%%%%%%%%%%%%%%%%%%%%%%%%%%%%%%%%%%%%%%%%%%%%%%%%%%%%%%%%%%%%%%%%%%%%%%%%%%%%%%%%%%%%%%%%%%%%%%%%%%%%%%%%%%%%%%%%%%%%%%%%%%%%%%%%

The preceding section gives the sharp worst case estimate in the Euclidean, or $L_2$, metric.
We now work with the moment metric $d_{X,p}$.
The purpose of this section is to prove the quantitative estimates summarized in Theorem~\ref{th:intro-SMP-summary} and to identify the ranges in which they improve the general bounds of Lata\l a.

\begin{theorem}\label{th:intro-Lp-SMP}
Let $X$ be isotropic and log-concave in $\R^n$, where $n\gr2$.
Then, for every $p\gr2$,
\begin{equation}\label{eq:intro-Lp-SMP}
X\text{ satisfies }\mathrm{SMP}_p\left(\frac{c}{\sqrt{\min\{p,n\}\ln(en)}}\right).
\end{equation}
Moreover, for every $n\gr1$, every nondegenerate $n$-dimensional log-concave random vector satisfies
\begin{equation}\label{eq:intro-full-SMP}
\mathrm{SMP}\left(\frac{c}{\sqrt{n\ln(en)}}\right).
\end{equation}
\end{theorem}

\begin{proof}
Assume first that $2\ls p\ls n$, and let $T\subseteq\R^n$ be finite with $|T|\gr e^p$ and
$$ \|\langle t-s,X\rangle\|_p\gr A \qquad(s,t\in T,\ s\ne t). $$
By~\eqref{eq:one-dimensional-moment-growth} and isotropy, $A\ls Cp\|\langle t-s,X\rangle\|_2=Cp|t-s|$.
Thus the points of $T$ are pairwise separated by at least $cA/p$.
For $\varepsilon=cA/p$, with a sufficiently small absolute constant $c>0$, $N(T,\varepsilon B_2^n)\gr |T|\gr e^p$.
Fix $t_0\in T$.
Since $X$ is centered, $\E\sup_{s,t\in T}\langle t-s,X\rangle \gr \E\sup_{t\in T}\langle t-t_0,X\rangle =\E\sup_{t\in T}\langle t,X\rangle$.
The comparison~\eqref{eq:gauge-comparison}, the Gaussian Sudakov estimate~\eqref{eq:Gaussian-Sudakov} and the preceding packing estimate give
$$ \E\sup_{s,t\in T}\langle t-s,X\rangle \gr\frac{cA}{\sqrt{p\ln(en)}}. $$
This proves~\eqref{eq:intro-Lp-SMP} for $2\ls p\ls n$.

Let $p>n$, let $X'$ be an independent copy of $X$, and put $Y=(X-X')/\sqrt2$.
Then $Y$ is symmetric, isotropic and log-concave.
By the preceding case, $Y$ satisfies $\mathrm{SMP}_n\left(\frac{c}{\sqrt{n\ln(en)}}\right)$.
Lemma~\ref{lem:Latala-tools}(3), followed by invariance under scalar multiplication, shows that $X-X'$ satisfies $\mathrm{SMP}_p\left(\frac{c}{\sqrt{n\ln(en)}}\right)$.
Lemma~\ref{lem:Latala-tools}(1) transfers this estimate to $X$.
This proves the first assertion.

Let now $Z$ be a nondegenerate $n$-dimensional log-concave random vector.
If $n=1$, Lemma~\ref{lem:Latala-tools}(4) gives the asserted global estimate after changing the absolute constant.
Assume $n\gr2$, let $Z'$ be an independent copy of $Z$, and let $\Sigma$ be the covariance matrix of $Z$.
The vector $W=\frac{1}{\sqrt2}\Sigma^{-1/2}(Z-Z')$ is isotropic, symmetric and log-concave.
By the first assertion and linear invariance, $Z-Z'$ satisfies $\mathrm{SMP}_p\left(\frac{c}{\sqrt{n\ln(en)}}\right),\; p\gr2$.
Lemma~\ref{lem:Latala-tools}(1) transfers this estimate to $Z$.
For $1\ls p\ls2$, Lemma~\ref{lem:Latala-tools}(4) gives an absolute constant, which is stronger.
Hence~\eqref{eq:intro-full-SMP} holds.
\end{proof}

\begin{proposition}\label{prop:ellstar-SMP}
Let $X$ be isotropic and log-concave in $\R^n$ and let $p\gr2$.
Then
$$ X\text{ satisfies }\mathrm{SMP}_p\left(\frac{cp}{\sqrt{\ln(en)}\,\ell^*(Z_p(X))}\right). $$
\end{proposition}

\begin{proof}
Let $T\subseteq\R^n$ be finite, $|T|\gr e^p$, and assume that
$$ \|\langle t-s,X\rangle\|_p\gr A \qquad(s,t\in T,\ s\ne t). $$
Put $K=\conv(T-T)$ and $I=\E h_K(X)=\E\sup_{s,t\in T}\langle t-s,X\rangle$.
Fix $t_0\in T$.
Since $\|t-s\|_{Z_p(X)^\circ}=h_{Z_p(X)}(t-s)=\|\langle t-s,X\rangle\|_p$,
the set $T-t_0\subseteq K$ is $(A/2)Z_p(X)^\circ$-separated.
Hence the packing and covering inequalities give
$$ \ln N\left(K,\frac A4 Z_p(X)^\circ\right)\gr\ln|T|\gr p. $$
For every $r>0$, submultiplicativity of covering numbers and \eqref{eq:classical-Sudakov-covering} yield
$$ \ln N\left(K,\frac A4 Z_p(X)^\circ\right) \ls \ln N(K,rB_2^n)+\ln N\left(rB_2^n,\frac A4 Z_p(X)^\circ\right) \ls C\left(\frac{\ell^*(K)^2}{r^2}+\frac{\ell^*(Z_p(X))^2r^2}{A^2}\right),$$
where we used $\ell(Z_p(X)^\circ)=\ell^*(Z_p(X))$.
Choosing $ r^2=\frac{A\ell^*(K)}{\ell^*(Z_p(X))} $ and using the preceding packing estimate, we obtain $ p\ls (C\ell^*(K)\ell^*(Z_p(X)))/A. $
The function $h_K$ is a gauge.
The reverse comparison \eqref{eq:gauge-comparison} therefore gives $\ell^*(K)=\E h_K(G)\ls C\sqrt{\ln(en)}\,I$.
Combining the last two estimates proves
$$ I\gr\frac{cAp}{\sqrt{\ln(en)}\,\ell^*(Z_p(X))}, $$
as required.
\end{proof}

Combining Proposition~\ref{prop:ellstar-SMP} with the preceding estimates, for $2\ls p\ls n$ the vector $X$ satisfies
$$ \mathrm{SMP}_p\left(c\max\left\{\frac1p,\frac{1}{\sqrt{p\ln(en)}},\frac{p}{\sqrt{\ln(en)}\,\ell^*(Z_p(X))}\right\}\right). $$
Lemma~\ref{lem:full-range-centroid-width} therefore gives the explicit estimate
$$ X\text{ satisfies }\mathrm{SMP}_p\left(c\max\left\{\frac1p,\frac1{\sqrt{p\ln(en)}},\frac{\sqrt p}{\sqrt{n\ln(en)\ln(e+p)}}\right\}\right), \qquad 2\ls p\ls n. $$
The third term is stronger than $1/\sqrt{p\ln(en)}$, the term obtained from the Gaussian Sudakov argument, when $p^2\gr n\ln(e+p)$.

\begin{proof}[Proof of Theorem~$\ref{th:intro-SMP-summary}$]
The estimate~\eqref{eq:intro-SMP-summary} is the maximum of the general bound in Lemma~\ref{lem:Latala-tools}(4), the estimate of Theorem~\ref{th:intro-Lp-SMP} and the explicit consequence of Proposition~\ref{prop:ellstar-SMP} obtained from Lemma~\ref{lem:full-range-centroid-width}.
The global estimate~\eqref{eq:intro-global-SMP-summary} is~\eqref{eq:intro-full-SMP}.
\end{proof}

For $p\gr C\ln(en)$, the term $1/\sqrt{p\ln(en)}$ improves the general $1/p$ estimate of Lata\l a.
The global estimate~\eqref{eq:intro-full-SMP} improves the general dimensional bound $c/n$ to $c/\sqrt{n\ln(en)}$.
The centroid body width term gives a further improvement once $p^2\gr Cn\ln(e+p)$.
On the other hand, for symmetric vectors Lata\l a's estimate~\eqref{eq:Latala-large-p} is absolute for $p\gr2n\ln(n+e)$, so the bounds are complementary in the large $p$ range.
None of these estimates gives the conjectural dimension free $\mathrm{SMP}(c)$ for all $p$.

A complementary dimension dependent estimate for very large separated families is given in Proposition~\ref{prop:large-family} in Appendix~\ref{appendix:additional-sudakov}.

%%%%%%%%%%%%%%%%%%%%%%%%%%%%%%%%%%%%%%%%%%%%%%%%%%%%%%%%%%%%%%%%%%%%%%%%%%%%%%%%%%%%%%%%%%%%%%%%%%%%%%%%%%%%%%%%%%%%%%%%%%%%%%%%%%%%%%
\section{Weak--strong moments and mean norms of centroid bodies}\label{sec:weak-strong}
%%%%%%%%%%%%%%%%%%%%%%%%%%%%%%%%%%%%%%%%%%%%%%%%%%%%%%%%%%%%%%%%%%%%%%%%%%%%%%%%%%%%%%%%%%%%%%%%%%%%%%%%%%%%%%%%%%%%%%%%%%%%%%%%%%%%%%

The estimates of the preceding section give different Sudakov constants at different moment scales.
The weak--strong problem asks whether strong moments of a norm of a log-concave vector can be controlled, up to an absolute constant, by its first moment and the corresponding weak moment.
In the proof of Theorem~\ref{th:intro-weak-strong} below, rather than inserting one global Sudakov constant into the chaining argument, we use at each dyadic level the strongest among the three bounds in~\eqref{eq:intro-SMP-summary}.
This is the point at which the estimates at different scales from Section~\ref{sec:Lp-SMP} produce a quantitative gain.

We shall use the following form of Lata\l a's chaining estimate~\cite[Proposition~6.1 and Remark~6.2]{Latala-Sudakov-2014}.
Let $(X_t)_{t\in T}$ be a real process.
Let $T_k\subseteq T$ satisfy $|T_k|\ls e^{2^{k+1}}$, let $\pi_k:T\to T_k$, and assume that $T_{k_1}=T$ and $\pi_{k_1}=\mathrm{id}$.
If $2^{k_0-1}\ls p\ls2^{k_0}$ and $k_0\ls k_1-1$, then
\begin{equation}\label{eq:Latala-chaining}
\left\|\sup_{t\in T}|X_t|\right\|_p \ls C\left(\sup_{t\in T}\sum_{k=k_0+1}^{k_1}\|X_{\pi_k(t)}-X_{\pi_{k-1}(t)}\|_{2^k} +\sup_{t\in T_{k_0}}\|X_t\|_p\right),
\end{equation}
while for $k_0\gr k_1$ the left hand side is at most $C\sup_{t\in T}\|X_t\|_p$.
A standard volumetric argument provides a $1/2$-net $\mathcal N$ of the dual unit ball, in the dual norm, with $|\mathcal N|\ls5^n$ and
\begin{equation}\label{eq:dual-net-norm}
\|y\|\ls2\max_{t\in\mathcal N}|\langle t,y\rangle|.
\end{equation}
Thus only the three scale dependent Sudakov estimates summarized in~\eqref{eq:intro-SMP-summary} enter the proof.

\begin{proof}[Proof of Theorem~$\ref{th:intro-weak-strong}$]
When $n=1$, every norm has the form $\|x\|=a|x|$ for some $a>0$, and hence the left hand side of~\eqref{eq:weak-strong-scale-dependent} is exactly $\sigma_p(Y)$.
We may therefore assume that $n\gr2$ and put $L=\ln(en)$.
We shall use the elementary bound $L\ls C\sqrt n$.
For $q\gr2$ define
\begin{equation}\label{eq:bq-definition}
b_q=\min\left\{ q,\;\sqrt{\min\{q,n\}L},\; \frac{\sqrt L\,\ell^*(Z_q(X))}{q} \right\}.
\end{equation}
Lemma~\ref{lem:Latala-tools}(4), Theorem~\ref{th:intro-Lp-SMP} and Proposition~\ref{prop:ellstar-SMP} imply that $X$ satisfies
\begin{equation}\label{eq:bq-SMP}
\mathrm{SMP}_q(c/b_q),\qquad q\gr2.
\end{equation}

Let $k_1$ be minimal with $2^{k_1+1}\gr n\ln5$.
We claim that
\begin{equation}\label{eq:bq-sum}
\sum_{k=1}^{k_1}b_{2^k}\ls C n^{1/4}\sqrt L\,\ln(e+L).
\end{equation}
For $2\ls q\ls\sqrt n$, the second term in~\eqref{eq:bq-definition} gives $b_q\ls\sqrt{qL}$, and hence
$$ \sum_{2^k\ls\sqrt n}b_{2^k}\ls Cn^{1/4}\sqrt L. $$
For $\sqrt n<q\ls\sqrt n\,L$, Borell's moment comparison and~\eqref{eq:Paouris-small-p-Mstar} give
$$ \ell^*(Z_q(X))\ls C\frac q{\sqrt n}\ell^*(Z_{\sqrt n}(X))\ls Cqn^{1/4}. $$
Therefore the third term in~\eqref{eq:bq-definition} yields $b_q\ls Cn^{1/4}\sqrt L$.
There are at most $C\ln(e+L)$ dyadic levels in this range, and their contribution is at most $Cn^{1/4}\sqrt L\,\ln(e+L)$.
For $\sqrt n\,L<q\ls n$, Lemma~\ref{lem:full-range-centroid-width} gives
$$ b_q\ls C\sqrt{\frac{nL\ln(e+q)}q}\ls C\frac{\sqrt n\,L}{\sqrt q}, $$
and the dyadic sum over this range is at most $Cn^{1/4}\sqrt L$.
Finally, if $n<q\ls2^{k_1}$, Borell's comparison with $Z_n(X)$ and Lemma~\ref{lem:full-range-centroid-width} give $b_q\ls CL$; this concerns only an absolute number of dyadic levels and is absorbed by the preceding bound.
This proves~\eqref{eq:bq-sum}.

Let $T$ be a $1/2$-net of the dual unit ball as in~\eqref{eq:dual-net-norm}; thus $|T|\ls5^n$ and $\|y\|\ls2\max_{t\in T}|\langle t,y\rangle|$.
Fix a sufficiently large absolute constant $C_0$.
For $q\gr2$, put $a_q=C_0b_q\,\E\|X\|$ and let $S_q\subseteq T$ be maximal $a_q$-separated for $d_{X,q}$.
If $|S_q|\gr e^q$, then~\eqref{eq:bq-SMP} gives
$$ 2\E\|X\|\gr\E\sup_{s,t\in S_q}\langle t-s,X\rangle \gr ca_q/b_q=cC_0\E\|X\|, $$
a contradiction for large $C_0$.
Hence $|S_q|<e^q$, and maximality makes $S_q$ an $a_q$-net of $T$ for $d_{X,q}$.

For $0\ls k<k_1$ set $ T_k=S_{2^{k+1}}, $ and set $T_{k_1}=T$.
By the preceding cardinality bound, $|T_k|\ls e^{2^{k+1}}$ for all $0\ls k\ls k_1$.
For $0\ls k<k_1$, choose a map $\pi_k:T\to T_k$ such that
$$ d_{X,2^{k+1}}(t,\pi_k(t))\ls a_{2^{k+1}}, \qquad t\in T, $$
and put $\pi_{k_1}=\mathrm{id}$.
For $1\ls k<k_1$, the triangle inequality, the choice of $\pi_k$ and monotonicity of moments give
\begin{equation}\label{eq:chain-increment-bound}
 d_{X,2^k}(\pi_k(t),\pi_{k-1}(t)) \ls d_{X,2^k}(\pi_k(t),t)+d_{X,2^k}(t,\pi_{k-1}(t)) \ls a_{2^{k+1}}+a_{2^k}.
\end{equation}
At the last level, the choice of $\pi_{k_1-1}$ gives $d_{X,2^{k_1}}(\pi_{k_1}(t),\pi_{k_1-1}(t)) \ls a_{2^{k_1}}$.
By~\eqref{eq:weak-domination}, the same bounds hold with $Y$ in place of $X$.

Let $k_0$ be the smallest positive integer such that $2^{k_0}\gr p$.
If $k_0\ls k_1-1$, apply~\eqref{eq:Latala-chaining} to the process $X_t=\langle t,Y\rangle$.
Using~\eqref{eq:chain-increment-bound}, the estimate at the last level, and the fact that $T_{k_0}$ is contained in the dual unit ball, we obtain
$$ \left\|\max_{t\in T}|\langle t,Y\rangle|\right\|_p \ls C\left(\sum_{k=k_0+1}^{k_1-1}(a_{2^k}+a_{2^{k+1}})+a_{2^{k_1}}+\sigma_p(Y)\right) \ls C\left(\sum_{k=1}^{k_1}a_{2^k}+\sigma_p(Y)\right).$$
By the definition of $a_q$ and~\eqref{eq:bq-sum},
$$ \sum_{k=1}^{k_1}a_{2^k} \ls C n^{1/4}\sqrt L\,\ln(e+L)\,\E\|X\|. $$
If $k_0\gr k_1$, the terminal case of~\eqref{eq:Latala-chaining} gives directly $\left\|\max_{t\in T}|\langle t,Y\rangle|\right\|_p \ls C\sigma_p(Y)$.
Combining the two cases with~\eqref{eq:dual-net-norm} proves~\eqref{eq:weak-strong-scale-dependent}.
Taking $Y=X$ gives~\eqref{eq:weak-strong-X}.
\end{proof}

If one uses only the constants from Theorem~\ref{th:intro-Lp-SMP} at the successive levels of the same chain, one obtains the weaker coefficient $C\sqrt{n\ln(en)}$.
Proposition~\ref{prop:ellstar-SMP}, together with the Paouris and Giannopoulos, Pafis and Tziotziou width estimates and Borell's comparison between consecutive moment scales, reduces this coefficient to $Cn^{1/4}\sqrt{\ln(en)}\,\ln(e+\ln(en))$.
The factor $\ln(e+\ln(en))$ comes from summing the intermediate dyadic range $\sqrt n<q\ls\sqrt n\ln(en)$.
We do not address whether this factor is necessary.
Removing it would require either a sharper full range estimate for $\ell^*(Z_q(X))$ in this range or a chaining argument that does not pay separately at every intermediate scale.
Lata\l a's result~\cite[Corollary~6.4]{Latala-Sudakov-2014} shows that $\mathrm{SMP}(\alpha)$ implies
$$ \left(\E\|X\|^p\right)^{1/p} \ls C\left(\frac1\alpha\max\left\{1,\ln\left(\frac{en}{p}\right)\right\}\E\|X\|+\sigma_p(X)\right). $$
Thus the global estimate~\eqref{eq:intro-full-SMP} alone would introduce an additional logarithmic factor.

\paragraph{Mean norms of \texorpdfstring{$L_p$}{Lp}-centroid bodies.} 
The weak--strong argument uses centroid bodies through their mean widths.
We now turn to the dual mean norm.
Earlier estimates were obtained by Giannopoulos, Stavrakakis, Tsolomitis and Vritsiou, who proved in~\cite[Theorem~7.2]{Giannopoulos-Stavrakakis-Tsolomitis-Vritsiou-2015} that
$$ M(Z_p(X))\ls C\frac{(\ln p)^{5/6}}{p^{1/6}},\qquad 2\ls p\ls n^{3/7}, $$
and by Giannopoulos and Milman, who proved in~\cite[Theorem~14]{Giannopoulos-Milman-2014} that
$$ M(Z_p(X))\ls C\frac{\sqrt{\ln p}}{p^{1/4}},\qquad 2\ls p\ls\bigl(n\ln(e+n)\bigr)^{2/5}. $$
These estimates can be stronger for some small values of $p$, whereas the estimate below holds throughout $2\ls p\ls n$ and reaches the endpoint $p=n$.
The result is a short consequence of the reverse gauge comparison and the theorem of Lata\l a and Nayar, and it serves as a bridge from the estimates for support functions above to the polar metrics used in the next section.

Very recently, Brazitikos and Pandis~\cite{Brazitikos-Pandis-2026} used a quadratic aggregate of dyadic centroid bodies to give a deterministic
geometric proof that $M(K)\ls C \ln(en)/\sqrt n$ for every origin symmetric convex body $K$ whose uniform probability measure is isotropic. 
Their result concerns the mean norm of the body itself and is methodologically distinct from the comparison between Gaussian and log-concave vectors used below.
Its numerical order is weaker than the optimal estimate $M(K)\ls C\sqrt{\ln(en)/n}$.

\begin{theorem}
For every isotropic log-concave random vector $X$ in $\R^n$ and every $p\gr2$,
\begin{equation}\label{eq:main-all-q}
M(Z_p(X))\ls C\sqrt{\frac{(n+p)\ln(en)}{np}}.
\end{equation}
In particular, for $2\ls p\ls n$,
\begin{equation}\label{eq:main-logarithmic}
\sqrt p\,M(Z_p(X))\ls C\sqrt{\ln(en)}.
\end{equation}
\end{theorem}

\begin{proof}
Apply the left hand inequality in~\eqref{eq:gauge-comparison} to the norm $\phi(x)=\|x\|_{Z_p(X)}$.
Then
$$ \E\|G\|_{Z_p(X)}\ls C\sqrt{\ln(en)}\E\|X\|_{Z_p(X)}. $$
By~\eqref{eq:LN-self-gauge} and H\"older's inequality,
$$ \E\|X\|_{Z_p(X)} \ls 2\sqrt e\sqrt{\frac{n+p}{p}}. $$
Moreover,~\eqref{eq:ell-M-relations} gives
$$ \E\|G\|_{Z_p(X)}\simeq\sqrt n\,M(Z_p(X)). $$
Combining the last three estimates proves~\eqref{eq:main-all-q}.
If $p\ls n$, then $(n+p)/n\ls2$, and~\eqref{eq:main-logarithmic} follows.
\end{proof}

\begin{corollary}
Let $X$ be isotropic and log-concave in $\R^n$.
For every $2\ls p\ls n$,
$$ M(Z_p(X))M^*(Z_p(X))\ls C\sqrt{\ln(en)\ln(e+p)}. $$
If $2\ls p\ls\sqrt n$, then
$$ M(Z_p(X))M^*(Z_p(X))\ls C\sqrt{\ln(en)}. $$
\end{corollary}

\begin{proof}
Multiply~\eqref{eq:main-logarithmic} by the corresponding estimate for $M^*(Z_p(X))$ from Lemma~\ref{lem:full-range-centroid-width}.
In the range $p\ls\sqrt n$, use~\eqref{eq:Paouris-small-p-Mstar} instead.
\end{proof}

The cube shows that the logarithmic factor in \eqref{eq:main-logarithmic} cannot be removed uniformly.
Indeed, if $X$ is uniform on the isotropic cube $Q_n=[-\sqrt3,\sqrt3]^n$, then $Z_n(X)\simeq Q_n$ (see, for example,~\cite{BGVV-book}) and $M(Q_n)\simeq\sqrt{\ln(en)/n}$, so
$$ \sqrt n\,M(Z_n(X))\gr c\sqrt{\ln(en)}. $$
The factor $\sqrt p$ is the natural normalization, since $Z_p(G)\simeq\sqrt p\,B_2^n$ for a standard Gaussian vector $G$.

The first part of the paper uses $M^*(Z_p(X))$, equivalently the Gaussian mean of the support function, to control Sudakov constants.
The preceding estimates treat the dual parameter $M(Z_p(X))$.
We now pass from these average parameters to the polar metrics $\mathcal I_r(X)Z_r(X)^\circ$, which form the second aspect of the same family of centroid bodies.

%%%%%%%%%%%%%%%%%%%%%%%%%%%%%%%%%%%%%%%%%%%%%%%%%%%%%%%%%%%%%%%%%%%%%%%%%%%%%%%%%%%%%%%%%%%%%%%%%%%%%%%%%%%%%%%%%%%%%%%%%%%%%%%%%%%%%%
\section{Entropy of polar centroid bodies}\label{sec:generalized}
%%%%%%%%%%%%%%%%%%%%%%%%%%%%%%%%%%%%%%%%%%%%%%%%%%%%%%%%%%%%%%%%%%%%%%%%%%%%%%%%%%%%%%%%%%%%%%%%%%%%%%%%%%%%%%%%%%%%%%%%%%%%%%%%%%%%%%

We now turn to the second aspect of the framework of centroid bodies in which the polar bodies $Z_r(X)^\circ$ provide the natural self generated metrics.
Mendelson, Milman and Paouris proposed the generalized dual Sudakov estimate
$$ \mathsf M\left(Z_p(\mu),C\left(\int_{\R^n}\|x\|_K\,d\mu(x)\right)K\right)\ls e^{Cp},\qquad p\gr1, $$
for an origin symmetric log-concave probability measure $\mu$ and an origin symmetric convex body $K$.
Their program first proves a dimension dependent weak estimate and then seeks a separation preserving dimension reduction, a one sided small ball analogue of the Johnson--Lindenstrauss lemma.
They establish this reduction for ellipsoids and, up to logarithmic losses, for cubes.
Here the target is the self generated body $Z_r(X)^\circ$.
Rather than constructing a new dimension reduction map, we compare this target with the covariance ellipsoid and apply their regular ellipsoidal entropy theorem.
The resulting Theorem~\ref{th:polar-centroid-entropy} is dimension free, but it retains explicit dependence on $r$.
At the conjectural scale $u=1$ it gives a bound of order $\exp(Cpr)$ in the isotropic case, while the $e^{Cp}$ packing size is obtained at the larger scale $u=\sqrt r$.
Thus the result is a dimension free statement for the self generated target, not a proof of the full Mendelson--Milman--Paouris conjecture.

\paragraph{The self generated polar metric}

Recall that for a random vector $X$ and an origin symmetric convex body $K$, we denote $I_1(X,K)=\E\|X\|_K$ and $I_1^*(X,K)=\E h_K(X)$ whenever these expectations are finite.
We also set $\mathcal I_r(X)=I_1(X,Z_r(X)^\circ)=\E h_{Z_r(X)}(X)$.

\begin{lemma}\label{lem:ellipsoidal-regular-entropy}
There is an absolute constant $C\gr1$ with the following property.
Let $X$ be a symmetric nondegenerate log-concave random vector in $\R^n$ and let $D$ be an origin symmetric ellipsoid.
Then, for every $p\gr1$ and every $t>0$,
\begin{equation}\label{eq:ellipsoidal-regular-entropy}
\ln\mathsf M\left(Z_p(X),CtI_1(X,D)D\right) \ls Cp\left(\frac1{t^2}+\frac1t\right).
\end{equation}
\end{lemma}

\begin{proof}
Using the notation of Mendelson, Milman and Paouris, for an origin symmetric convex body $K$ and a log-concave probability measure $\eta$, put
$$ m_1(\eta,K)=\sup\{a>0:\eta(aK)\ls e^{-1}\},\qquad I_1(\eta,K)=\int_{\R^n}\|x\|_K\,d\eta(x). $$
Thus $m_1(\eta,K)$ is the $e^{-1}$ quantile scale of the gauge $\|\cdot\|_K$, while $I_1(\eta,K)$ is its mean.
Let $\mu$ be the law of $X$.
Mendelson, Milman and Paouris~\cite[Lemma~2.1]{Mendelson-Milman-Paouris} proved that $m_1(\eta,K)\simeq I_1(\eta,K)$.
Their regular entropy estimate~\cite[Theorem~7.4]{Mendelson-Milman-Paouris} gives
$$ \ln\mathsf M\left(Z_p(X),t\,m_1(\mu,D)D\right) \ls Cp\left(\frac1{t^2}+\frac1t\right), $$
up to an absolute factor arising from their packing convention based on disjoint translates.
The comparison of $m_1$ and $I_1$ proves~\eqref{eq:ellipsoidal-regular-entropy}.
\end{proof}

For uniform measures on isotropic convex bodies, Giannopoulos, Paouris and Vritsiou~\cite[Proposition~3.2]{Giannopoulos-Paouris-Vritsiou-2012} proved, in particular, the lower estimate $I_1(K,Z_q(K)^\circ)\gr c\sqrt{nq}$ for $1\ls q\ls n,$ and raised the question of a reverse estimate of the same order.
Skarmogiannis~\cite[Theorem~1.6]{Skarmogiannis-2023} subsequently obtained the desired $\sqrt q$ dependence up to logarithmic factors; more precisely, if $\sigma_n\ls C(\ln(en))^\gamma$, his argument gives an upper bound with a factor $L_K^2(\ln(en))^{\gamma+3}$.
Using the current slicing bound together with the first moment comparison~\eqref{eq:gauge-comparison} and Lemma~\ref{lem:full-range-centroid-width}, this logarithmic loss can be reduced to
$ I_1(K,Z_q(K)^\circ)\ls C\sqrt{nq\ln(en)\ln(e+q)}$ for $1\ls q\ls n. $
Indeed, if $Y$ is uniformly distributed on $K$ and $X=Y/L_K$, then $ I_1(K,Z_q(K)^\circ)=L_K^2\mathcal I_q(X), $ while $ \mathcal I_q(X)\ls C\sqrt{\ln(en)}\,\ell^*(Z_q(X))\ls C\sqrt{nq\ln(en)\ln(e+q)}. $
For $q\ls\sqrt n$, Paouris' estimate~\eqref{eq:Paouris-small-p-Mstar} gives the sharper bound $ I_1(K,Z_q(K)^\circ)\ls C\sqrt{nq\ln(en)}. $
Lemma~\ref{lem:Iq-lower-bound} records the lower estimate in the more general setting of arbitrary isotropic log-concave vectors.

\begin{lemma}\label{lem:Iq-lower-bound}
Let $X$ be isotropic and log-concave in $\R^n$.
Then
\begin{equation}\label{eq:Iq-lower-bound}
\mathcal I_q(X)\gr c\sqrt{nq},\qquad 2\ls q\ls n.
\end{equation}
Consequently,
\begin{equation}\label{eq:Iq-ratio}
\frac{\mathcal I_2(X)}{\mathcal I_q(X)}\ls\frac{C}{\sqrt q},\qquad 2\ls q\ls n.
\end{equation}
\end{lemma}

\begin{proof}
Let $f$ be the density of $X$, put $L_X=\|f\|_\infty^{1/n}$ and set $K=Z_q(X)$.
By Markov's inequality and the definition of $\mathcal I_q(X)$, 
$$\mathbb P\{X\in 2\mathcal I_q(X)K^\circ\}\gr\frac12.$$
It follows that $|K^\circ|^{1/n}\gr\frac{1}{4\mathcal I_q(X)L_X}$.
The Blaschke--Santal\'o inequality and $|B_2^n|^{2/n}\ls C/n$ give $|K^\circ|^{1/n}\ls\frac{C}{n|K|^{1/n}}$, and hence
$$ \mathcal I_q(X)\gr\frac{cn}{L_X}|Z_q(X)|^{1/n}. $$
The standard lower volume estimate for centroid bodies,
$$ |Z_q(X)|^{1/n}\gr\frac{c}{L_X}\sqrt{\frac qn}, \qquad 1\ls q\ls n, $$
may be found, for example, in~\cite{BGVV-book}.
The affirmative solution of the slicing problem gives $L_X\ls C$ uniformly~\cite{Klartag-Lehec-2025}.
Substitution in the preceding estimate proves~\eqref{eq:Iq-lower-bound}.
Finally, $Z_2(X)=B_2^n$ and therefore $\mathcal I_2(X)=\E|X|\ls\sqrt n$; this proves~\eqref{eq:Iq-ratio}.
\end{proof}

\begin{proof}[Proof of Theorem~$\ref{th:polar-centroid-entropy}$]
Suppose first that $X$ is origin symmetric.
Put $D=Z_2(X)^\circ$.
The one-dimensional moment estimate~\eqref{eq:one-dimensional-moment-growth} is equivalent to $Z_r(X)\subseteq CrZ_2(X)$, and hence $Z_r(X)^\circ\supseteq\frac{c}{r}D$.
Apply Lemma~\ref{lem:ellipsoidal-regular-entropy} with $t=\frac{cu\mathcal I_r(X)}{r\mathcal I_2(X)}$.
After adjusting the absolute constant in the packing scale, the preceding inclusion and monotonicity of packing numbers give
$$ \ln\mathsf M\left(Z_p(X),Cu\mathcal I_r(X)Z_r(X)^\circ\right) \ls \ln\mathsf M\left(Z_p(X),ct\mathcal I_2(X)D\right) \ls Cp\left[ \left(\frac{r\mathcal I_2(X)}{u\mathcal I_r(X)}\right)^2+ \frac{r\mathcal I_2(X)}{u\mathcal I_r(X)} \right].$$

We pass to a centered, not necessarily symmetric, vector.
Let $X'$ be an independent copy of $X$ and put $Y=X-X'$.
For every $q\gr1$,
\begin{equation}\label{eq:symmetrization-centroid-relations}
Z_q(X)\subseteq Z_q(Y)\subseteq2Z_q(X).
\end{equation}
Moreover, conditional Jensen's inequality and the triangle inequality imply
$$ \mathcal I_q(X)\ls \mathcal I_q(Y)\ls4\mathcal I_q(X), $$
while $Z_2(Y)=\sqrt2 Z_2(X)$ gives the sharper upper estimate
$$ \mathcal I_2(Y)\ls2\sqrt2 \mathcal I_2(X). $$
Indeed,
$$ \mathcal I_q(Y)\gr\E h_{Z_q(X)}(X-X')\gr\E h_{Z_q(X)}(X)=\mathcal I_q(X), $$
and the two upper estimates follow directly from~\eqref{eq:symmetrization-centroid-relations} and subadditivity of support functions.

Apply the symmetric estimate to $Y$.
The preceding symmetrization relations show that
$$ Cu\mathcal I_r(Y)Z_r(Y)^\circ\subseteq C'u\mathcal I_r(X)Z_r(X)^\circ $$
and
$$ \frac{\mathcal I_2(Y)}{\mathcal I_r(Y)}\ls C\frac{\mathcal I_2(X)}{\mathcal I_r(X)}. $$
Since $Z_p(X)\subseteq Z_p(Y)$, monotonicity of packing numbers proves \eqref{eq:polar-centroid-entropy} after changing the absolute constant.
Because $Z_2(X)\subseteq Z_r(X)$, one has $\mathcal I_2(X)\ls \mathcal I_r(X)$, and \eqref{eq:polar-centroid-entropy-universal} follows.
If $X$ is isotropic and $r\ls n$, Lemma~\ref{lem:Iq-lower-bound} gives $\mathcal I_2(X)/\mathcal I_r(X)\ls C/\sqrt r$ and yields \eqref{eq:polar-centroid-entropy-isotropic}.
\end{proof}

For comparison, Proposition~\ref{prop:MMP-entropy} in Appendix~\ref{appendix:MMP-comparison} records the dimension dependent bounds obtained directly from the Mendelson--Milman--Paouris program.
In the isotropic range $2\ls r\ls n$, our estimate is dimension free and gives
$$ \ln\mathsf M\left(Z_p(X),Cu\mathcal I_r(X)Z_r(X)^\circ\right)\ls Cp\left(\frac r{u^2}+\frac{\sqrt r}{u}\right). $$
The Mendelson--Milman--Paouris estimate has no explicit dependence on $r$ beyond the normalization of the target, but retains the ambient dimension.
Neither estimate dominates the other for all values of $p,r$ and $u$.
The complete formulas and the comparison of the available parameter ranges are given in Appendix~\ref{appendix:MMP-comparison}.

\paragraph{Factorization consequences.}
The remaining results of this section are factorization consequences.
They are obtained by passing through the Euclidean ball and applying the classical Sudakov and dual Sudakov inequalities.
The point is not the optimization itself, but that the resulting bounds hold for an arbitrary symmetric convex body $K$ and can then be specialized back to polar centroid bodies.

The comparison \eqref{eq:gauge-comparison} and \eqref{eq:ell-M-relations} imply
\begin{equation}\label{eq:gauge-I-to-ell}
\ell(K)\ls C\sqrt{\ln(en)}I_1(X,K),\qquad \ell^*(K)\ls C\sqrt{\ln(en)}I_1^*(X,K).
\end{equation}

\begin{proposition}\label{prop:body-dependent-factorization}
Let $X$ be isotropic and log-concave in $\R^n$, let $K$ be an origin symmetric convex body and let $p\gr2$.
Then, for every $t>0$,
\begin{align}
 \ln N(Z_p(X),tI_1(X,K)K) &\ls \frac{C\ell^*(Z_p(X))}{t}\, \frac{\ell(K)}{I_1(X,K)},\label{eq:generalized-dual-body-dependent}\\
 \ln N(K,tI_1^*(X,K)Z_p(X)^\circ) &\ls \frac{C\ell^*(Z_p(X))}{t}\, \frac{\ell^*(K)}{I_1^*(X,K)}.\label{eq:generalized-primal-body-dependent}
\end{align}
Consequently,
\begin{equation}\label{eq:generalized-universal-Mstar}
\max\left\{ \ln N(Z_p(X),tI_1(X,K)K), \ln N(K,tI_1^*(X,K)Z_p(X)^\circ) \right\} \ls \frac{C\sqrt{n\ln(en)}\,M^*(Z_p(X))}{t}.
\end{equation}
\end{proposition}

\begin{proof}
Fix $s>0$.
By submultiplicativity of covering numbers,
$$ N(Z_p(X),tI_1(X,K)K) \ls N(Z_p(X),sB_2^n)\, N(sB_2^n,tI_1(X,K)K). $$
By \eqref{eq:classical-Sudakov-covering},
$$ \ln N(Z_p(X),sB_2^n) \ls \frac{C\ell^*(Z_p(X))^2}{s^2}, $$
whereas
$$ \ln N(sB_2^n,tI_1(X,K)K) =\ln N\left(B_2^n,\frac{tI_1(X,K)}sK\right) \ls \frac{C\ell(K)^2s^2}{t^2I_1(X,K)^2}. $$
Hence
$$ \ln N(Z_p(X),tI_1(X,K)K) \ls C\left(\frac{\ell^*(Z_p(X))^2}{s^2} +\frac{\ell(K)^2s^2}{t^2I_1(X,K)^2}\right). $$
Choosing $s^2=\frac{tI_1(X,K)\ell^*(Z_p(X))}{\ell(K)}$ gives \eqref{eq:generalized-dual-body-dependent}.

For the primal estimate, again by submultiplicativity,
$$ N(K,tI_1^*(X,K)Z_p(X)^\circ) \ls N(K,sI_1^*(X,K)B_2^n)\, N(sI_1^*(X,K)B_2^n,tI_1^*(X,K)Z_p(X)^\circ). $$
Again by \eqref{eq:classical-Sudakov-covering},
$$ \ln N(K,sI_1^*(X,K)B_2^n) \ls \frac{C\ell^*(K)^2}{s^2I_1^*(X,K)^2}. $$
Moreover,
$$ \ln N(sI_1^*(X,K)B_2^n,tI_1^*(X,K)Z_p(X)^\circ)=\ln N\left(B_2^n,\frac ts Z_p(X)^\circ\right). $$
Since
$$ \ell(Z_p(X)^\circ)=\E\|G\|_{Z_p(X)^\circ}=\E h_{Z_p(X)}(G)=\ell^*(Z_p(X)), $$
\eqref{eq:classical-Sudakov-covering} yields
$$ \ln N\left(B_2^n,\frac ts Z_p(X)^\circ\right) \ls \frac{C\ell^*(Z_p(X))^2s^2}{t^2}. $$
Optimizing in $s$ gives \eqref{eq:generalized-primal-body-dependent}.
Finally, \eqref{eq:generalized-universal-Mstar} follows from \eqref{eq:gauge-I-to-ell} and \eqref{eq:ell-M-relations}.
\end{proof}

\begin{corollary}\label{cor:centroid-width-entropy}
Let $X$ be isotropic and log-concave in $\R^n$ and let $2\ls p,r\ls n$.
Then, for every $u>0$,
\begin{equation}\label{eq:centroid-width-covering}
\ln N\left(Z_p(X),u\mathcal I_r(X)Z_r(X)^\circ\right) \ls \frac{C\sqrt{np\ln(e+p)\ln(e+r)}}{u}.
\end{equation}
Consequently,
$$ \ln\mathsf M\left(Z_p(X),Cu\mathcal I_r(X)Z_r(X)^\circ\right) \ls \frac{C\sqrt{np\ln(e+p)\ln(e+r)}}{u}. $$
If $p\ls\sqrt n$, the factor $\ln(e+p)$ may be omitted; if $r\ls\sqrt n$, the factor $\ln(e+r)$ may be omitted.
In particular, both factors may be omitted when $p,r\ls\sqrt n$.
\end{corollary}

\begin{proof}
Apply Proposition~\ref{prop:body-dependent-factorization} with $K=Z_r(X)^\circ$, using~\eqref{eq:generalized-dual-body-dependent}.
Since $I_1(X,K)=\mathcal I_r(X)$ and $\ell(K)=\ell^*(Z_r(X))$,
Lemma~\ref{lem:Iq-lower-bound} and Lemma~\ref{lem:full-range-centroid-width} give
$$ \frac{\ell(K)}{I_1(X,K)} =\frac{\ell^*(Z_r(X))}{\mathcal I_r(X)} \ls C\sqrt{\ln(e+r)}. $$
A second application of Lemma~\ref{lem:full-range-centroid-width}, now with $q=p$, proves~\eqref{eq:centroid-width-covering}.
The packing estimate follows from the packing and covering inequalities.
Finally, Paouris' estimate~\eqref{eq:Paouris-small-p-Mstar} removes the corresponding logarithmic factor whenever $p\ls\sqrt n$ or $r\ls\sqrt n$.
\end{proof}

The three available estimates and the regimes in which they are effective are compared in Appendix~\ref{appendix:MMP-comparison}.

\begin{corollary}\label{cor:full-range-body-factorization}
Let $X$ be isotropic and log-concave in $\R^n$, let $K$ be origin symmetric and let $2\ls p\ls n$.
Then, for every $t>0$,
$$ \max\left\{ \ln N(Z_p(X),tI_1(X,K)K), \ln N(K,tI_1^*(X,K)Z_p(X)^\circ) \right\}  \ls \frac{C\sqrt{np\ln(en)\ln(e+p)}}{t}.$$
Consequently,
\begin{align*}
 \mathsf M\left(Z_p(X),C\sqrt{\frac{n\ln(en)\ln(e+p)}p}\,I_1(X,K)K\right)&\ls e^{Cp},\\
 \mathsf M\left(K,C\sqrt{\frac{n\ln(en)\ln(e+p)}p}\,I_1^*(X,K)Z_p(X)^\circ\right)&\ls e^{Cp}.
\end{align*}
If $p\ls\sqrt n$, the factor $\ln(e+p)$ may be omitted in all three estimates.
\end{corollary}

\begin{proof}
Combine Proposition~\ref{prop:body-dependent-factorization}, in the form~\eqref{eq:generalized-universal-Mstar}, with Lemma~\ref{lem:full-range-centroid-width}.
For $p\ls\sqrt n$, use~\eqref{eq:Paouris-small-p-Mstar}.
The packing estimates follow from the packing and covering inequalities.
\end{proof}

%%%%%%%%%%%%%%%%%%%%%%%%%%%%%%%%%%%%%%%%%%%%%%%%%%%%%%%%%%%%%%%%%%%%%%%%%%%%%%%%%%%%%%%%%%%%%%%%%%%%%%%%%%%%%%%%%%%%%%%%%%%%%%%%%%%%%%
\section{Affine dimensional entropy and applications to convex bodies}\label{sec:affine-applications}
%%%%%%%%%%%%%%%%%%%%%%%%%%%%%%%%%%%%%%%%%%%%%%%%%%%%%%%%%%%%%%%%%%%%%%%%%%%%%%%%%%%%%%%%%%%%%%%%%%%%%%%%%%%%%%%%%%%%%%%%%%%%%%%%%%%%%%

The preceding entropy theorem controls the size of a separated family but does not use the dimension of its affine hull.
The main point of this section is to recover such a dependence.
There are two inputs.
The projection geometry of the subgaussian body supplies a large section of the polar target inside the affine hull, while Theorem~\ref{th:polar-centroid-entropy} localizes a large part of the separated family to bounded $L_2$ diameter.
Their combination leads to Theorem~\ref{th:affine-dimension-entropy}.
The propositions before that theorem isolate these two ingredients; the final convex body calculation is only an application.

\paragraph{Projection geometry and localization.}
Following Giannopoulos, Pafis and Tziotziou~\cite{Giannopoulos-Pafis-Tziotziou-2026}, we call the body below the subgaussian body of $X$.
The normalization by $\sqrt q$ is the Gaussian normalization, since $Z_q(G)\simeq\sqrt q B_2^n$.
The next proposition follows from their volume theorem and the argument used in~\cite[Theorem~1.3]{Giannopoulos-Pafis-Tziotziou-2026}.

\begin{proposition}\label{prop:Psi-projection}
Let $X$ be isotropic and log-concave in $\R^n$, and put
$$ \Psi_{2,n}(X)=\conv\left\{\frac{Z_q(X)}{\sqrt q}:1\ls q\ls n\right\}. $$
If $E\subseteq\R^n$ has dimension $m$, then
\begin{equation}\label{eq:Psi-projection-inclusion}
P_E\Psi_{2,n}(X)
\subseteq C\sqrt{\frac nm}\,\Psi_{2,m}(P_EX)
\end{equation}
and
\begin{equation}\label{eq:Psi-projection-volume}
\operatorname{vrad}_E\left(P_E\Psi_{2,n}(X)\right) \ls C\sqrt{\frac nm}.
\end{equation}
\end{proposition}

\begin{proof}
Write $Y=P_EX$.
The projection identity for centroid bodies gives $P_EZ_q(X)=Z_q(Y)$.
For $q\ls m$, the body $Z_q(Y)/\sqrt q$ is contained in $\Psi_{2,m}(Y)$.
For $m<q\ls n$, Borell's moment comparison gives
$$ \frac{Z_q(Y)}{\sqrt q} \subseteq C\sqrt{\frac qm}\,\frac{Z_m(Y)}{\sqrt m} \subseteq C\sqrt{\frac nm}\,\Psi_{2,m}(Y). $$
Taking convex hulls proves~\eqref{eq:Psi-projection-inclusion}.

The volume theorem of Giannopoulos, Pafis and Tziotziou~\cite[Theorem~3.2]{Giannopoulos-Pafis-Tziotziou-2026} states that the subgaussian body of a centered log-concave measure has bounded volume ratio with respect to its covariance ellipsoid.
Since $Y$ is isotropic in $E$, this gives $\operatorname{vrad}_E\left(\Psi_{2,m}(Y)\right)\ls C$,
and~\eqref{eq:Psi-projection-volume} follows.
\end{proof}

The polar consequences needed below are the following.

\begin{corollary}\label{cor:Psi-polar}
Under the assumptions of Proposition~$\ref{prop:Psi-projection}$,
\begin{equation}\label{eq:Psi-polar-section}
\operatorname{vrad}_E\left(\Psi_{2,n}(X)^\circ\cap E\right) \gr c\sqrt{\frac mn}.
\end{equation}
Moreover, for every $2\ls r\ls n$,
\begin{equation}\label{eq:Psi-target-inclusion}
\frac{\mathcal I_r(X)}{\sqrt r}\,\Psi_{2,n}(X)^\circ \subseteq \mathcal I_r(X)Z_r(X)^\circ,
\end{equation}
and
\begin{equation}\label{eq:Psi-sections}
\operatorname{vrad}_E\left( \frac{\mathcal I_r(X)}{\sqrt r}\bigl(\Psi_{2,n}(X)^\circ\cap E\bigr) \right) \gr c\sqrt m.
\end{equation}
\end{corollary}

\begin{proof}
We have $\left(P_E\Psi_{2,n}(X)\right)^\circ =\Psi_{2,n}(X)^\circ\cap E$, where the polar on the left is taken in $E$.
The reverse Santal\'o inequality and~\eqref{eq:Psi-projection-volume} prove~\eqref{eq:Psi-polar-section}.

Finally, $Z_r(X)/\sqrt r\subseteq\Psi_{2,n}(X)$, and polarity gives~\eqref{eq:Psi-target-inclusion}.
Combining~\eqref{eq:Psi-polar-section} with $\mathcal I_r(X)/\sqrt r\gr c\sqrt n$, which follows from Lemma~\ref{lem:Iq-lower-bound}, proves~\eqref{eq:Psi-sections}.
\end{proof}

The following proposition localizes a separated subset of $Z_p(X)$ in the metric $d_{X,2}(u,v)=\|\langle u-v,X\rangle\|_2$ determined by the covariance of $X$.

\begin{proposition}\label{prop:packing-localization}
Let $X$ be centered, nondegenerate and log-concave in $\R^n$, let $p\gr1$, $r\gr2$ and $a>0$, and let $T\subseteq Z_p(X)$ be finite.
Assume that
\begin{equation}\label{eq:packing-separation}
\|\langle u-v,X\rangle\|_r\gr a\mathcal I_r(X) \qquad(u,v\in T,\ u\ne v).
\end{equation}
Then there is $T'\subseteq T$ such that
\begin{equation}\label{eq:packing-cardinality}
|T'|\gr |T|\exp(-Cp)
\end{equation}
and
\begin{equation}\label{eq:packing-L2-diameter}
\|\langle u-v,X\rangle\|_2\ls C\mathcal I_2(X) \qquad(u,v\in T').
\end{equation}
In particular, all distinct $u,v\in T'$ satisfy
\begin{equation}\label{eq:packing-moment-ratio}
\frac{\|\langle u-v,X\rangle\|_r} {\|\langle u-v,X\rangle\|_2} \gr ca\,\frac{\mathcal I_r(X)}{\mathcal I_2(X)}.
\end{equation}
If $X$ is isotropic and $2\ls r\ls n$, the right hand side of \eqref{eq:packing-moment-ratio} is at least $ca\sqrt r$.
\end{proposition}

\begin{proof}
Theorem~\ref{th:polar-centroid-entropy}, applied with $r=2$ and $u=1$, and the packing and covering inequalities give
$$ N\left(Z_p(X),C\mathcal I_2(X)Z_2(X)^\circ\right)\ls e^{Cp}. $$
Partition $T$ according to such a cover and choose a cell of maximal cardinality.
Its intersection $T'$ with $T$ satisfies~\eqref{eq:packing-cardinality}, and two points in the same cell satisfy~\eqref{eq:packing-L2-diameter}, after adjusting $C$.
Combining this with~\eqref{eq:packing-separation} proves~\eqref{eq:packing-moment-ratio}.
The isotropic assertion follows from Lemma~\ref{lem:Iq-lower-bound}.
\end{proof}

\paragraph{The affine dimensional bound.}
We now combine the projection input with the localization proposition.
The first term in the resulting estimate comes from elementary Euclidean packing, while the second uses the large section supplied by the subgaussian body.
This distinction is useful when comparing the two parameter regimes.

\begin{theorem}\label{th:affine-dimension-entropy}
There is an absolute constant $C\gr1$ with the following property.
Let $X$ be isotropic and log-concave in $\R^n$, let $p\gr1$, $2\ls r\ls n$, $a>0$, and let $T\subseteq Z_p(X)$ be nonempty and finite.
Assume that
\begin{equation}\label{eq:affine-dimension-separation}
 \|\langle u-v,X\rangle\|_r\gr a\mathcal I_r(X)  \qquad(u,v\in T,\ u\ne v).
\end{equation}
Put $d=\dim\operatorname{aff}T$.
If $d=0$, then $|T|=1$.
If $d\gr1$, then
\begin{equation}\label{eq:affine-dimension-entropy}
 \ln|T| \ls Cp+Cd\ln\left[ e+C\min\left\{ \frac{r\mathcal I_2(X)}{a\mathcal I_r(X)}, \frac{\mathcal I_2(X)}{a\sqrt d} +\frac{\mathcal I_r(X)}{\sqrt{rd}} \right\} \right].
\end{equation}
Moreover,
\begin{equation}\label{eq:affine-dimension-entropy-isotropic}
 \ln|T| \ls Cp+Cd\ln\left[ e+C\min\left\{ \frac{\sqrt r}{a}, \left(\frac1a+\min\left\{\sqrt r,\sqrt{\ln(en)\ln(e+r)}\right\}\right)\sqrt{\frac nd} \right\} \right]
\end{equation}
for every $d\gr1$.
If $r\ls\sqrt n$, the term $\sqrt{\ln(en)\ln(e+r)}$ in~\eqref{eq:affine-dimension-entropy-isotropic} may be replaced by $\sqrt{\ln(en)}$.
\end{theorem}

\begin{proof}
If $d=0$, then $T$ is a singleton.
Assume that $d\gr1$ and apply Proposition~\ref{prop:packing-localization}.
There is $T'\subseteq T$ such that
$$ |T'|\gr|T|\exp(-Cp), \qquad \|\langle u-v,X\rangle\|_2\ls C\mathcal I_2(X) \quad(u,v\in T'). $$
If $|T'|=1$, then $|T|\ls e^{Cp}$ and the result follows.
Otherwise fix $t_0\in T'$, put $S=T'-t_0$, and let $E=\operatorname{span}S,\; m=\dim E$.
Then $1\ls m\ls d$, and isotropy gives $S\subseteq C\mathcal I_2(X)B_2^E$.

We first use only Euclidean separation.
By~\eqref{eq:one-dimensional-moment-growth} and~\eqref{eq:affine-dimension-separation}, distinct $u,v\in T'$ satisfy
$$ |u-v|=\|\langle u-v,X\rangle\|_2 \gr \frac{ca\mathcal I_r(X)}{r}. $$
Comparing the volumes of disjoint Euclidean balls centered at the points of $S$ and using this containment, we obtain
$$ |T'|\ls \left[e+C\frac{r\mathcal I_2(X)}{a\mathcal I_r(X)}\right]^m. $$
Consequently,
\begin{equation}\label{eq:affine-dimension-elementary}
 \ln|T|\ls Cp+Cd\ln\left[e+C\frac{r\mathcal I_2(X)}{a\mathcal I_r(X)}\right].
\end{equation}

We next use the projection estimate.
Set
$$ D_E=\frac{a\mathcal I_r(X)}{2\sqrt r} \bigl(\Psi_{2,n}(X)^\circ\cap E\bigr). $$
Since $Z_r(X)/\sqrt r\subseteq\Psi_{2,n}(X)$, polarity gives $D_E\subseteq\frac a2\mathcal I_r(X)Z_r(X)^\circ$.
Thus~\eqref{eq:affine-dimension-separation} implies that the points of $S$ are $D_E$-separated.
Corollary~\ref{cor:Psi-polar} gives $\operatorname{vrad}_E(D_E)\gr ca\sqrt m$.
Moreover, $Z_2(X)/\sqrt2\subseteq\Psi_{2,n}(X)$ and $Z_2(X)=B_2^n$, hence $D_E\subseteq C\frac{a\mathcal I_r(X)}{\sqrt r}B_2^E$.
The translates $x+\frac12D_E$, $x\in S$, have pairwise disjoint interiors and are contained in
$$ \left(C\mathcal I_2(X)+C\frac{a\mathcal I_r(X)}{\sqrt r}\right)B_2^E. $$
Using the preceding lower bound for $\operatorname{vrad}_E(D_E)$ and comparing volumes in $E$, we obtain
$$ |T'|\ls \left[ C\frac{\mathcal I_2(X)}{a\sqrt m} +C\frac{\mathcal I_r(X)}{\sqrt{rm}} \right]^m. $$
Therefore
$$ \ln|T|\ls Cp+Cm\ln\left[ e+C\frac{\mathcal I_2(X)}{a\sqrt m} +C\frac{\mathcal I_r(X)}{\sqrt{rm}} \right]. $$
For $A\gr0$, the function $x\mapsto x\ln(e+A/\sqrt x)$ is increasing on $(0,\infty)$, since, with $t=A/\sqrt x$,
$$ \frac{d}{dx}\left[x\ln\left(e+\frac{A}{\sqrt x}\right)\right] =\ln(e+t)-\frac{t}{2(e+t)}>0. $$
Since $m\ls d$, the last estimate and~\eqref{eq:affine-dimension-elementary} give~\eqref{eq:affine-dimension-entropy}.

Finally, Lemma~\ref{lem:Iq-lower-bound} gives $\frac{r\mathcal I_2(X)}{a\mathcal I_r(X)}\ls C\frac{\sqrt r}{a}$.
Also, $\mathcal I_2(X)\ls\sqrt n$.
Borell's moment comparison gives $\mathcal I_r(X)\ls Cr\sqrt n$, while the upper first moment comparison~\eqref{eq:gauge-comparison} and Lemma~\ref{lem:full-range-centroid-width} give
$$ \mathcal I_r(X)=\E h_{Z_r(X)}(X) \ls C\sqrt{\ln(en)}\,\ell^*(Z_r(X)) \ls C\sqrt{nr\ln(en)\ln(e+r)}. $$
Using the better of these two estimates in~\eqref{eq:affine-dimension-entropy} proves~\eqref{eq:affine-dimension-entropy-isotropic}.
If $r\ls\sqrt n$, Paouris' estimate~\eqref{eq:Paouris-small-p-Mstar} allows $\sqrt{\ln(en)\ln(e+r)}$ to be replaced by $\sqrt{\ln(en)}$.
\end{proof}

\begin{corollary}
Under the assumptions of Theorem~$\ref{th:affine-dimension-entropy}$, let $d=\dim\operatorname{aff}T$.
If $d\ln\left(e+C\sqrt r/a\right)\ls cp,$ then $|T|\ls\exp(Cp)$.
More generally, if $d\gr1$, the same conclusion holds if
$$ d\ln\left[ e+C\min\left\{ \frac{\sqrt r}{a}, \left(\frac1a+\min\left\{\sqrt r,\sqrt{\ln(en)\ln(e+r)}\right\}\right)\sqrt{\frac nd} \right\} \right]\ls cp. $$
If $r\ls\sqrt n$, the second condition remains valid with $\sqrt{\ln(en)\ln(e+r)}$ replaced by $\sqrt{\ln(en)}$.
\end{corollary}

Let $K\subseteq\R^n$ be origin symmetric, isotropic and of volume one, let $Y$ be uniformly distributed on $K$, and put $X=Y/L_K$ and $\widetilde K=K/L_K$.
Then $X$ is isotropic, $\widetilde K$ is isotropic in the probabilistic normalization, and $c\ls L_K\ls C$, where the upper bound follows from the slicing theorem~\cite{Klartag-Lehec-2025}.
Tziotziou~\cite[Proposition~4.2]{Tziotziou-2026} proved that, with $r_n=|B_2^n|^{-1/n}\simeq\sqrt n$, one may choose $\delta_n\ls C\ln(en)$ so that
$$N(r_nB_2^n,tK)\ls\exp\left(\frac{\delta_n^2n}{t^2}\right),\qquad t>0.$$
Her proof gives more precisely $\delta_n\ls Cr_nM(K)\simeq C\sqrt n\,M(K).$
Since $\widetilde K$ is isotropic in the probabilistic normalization, Bizeul~\cite[Theorem~1.4]{Bizeul-MMstar-2026} gives $M(\widetilde K)\ls C\sqrt{\ln(en)/n}.$
Since $M(\widetilde K)=L_KM(K)$ and $L_K\gr c$, it follows that $M(K)\ls C\sqrt{\ln(en)/n}.$
Using this sharper input in Tziotziou's proof improves the parameter to $\delta_n\ls C\sqrt{\ln(en)}.$
We use this updated form of Tziotziou's covering argument below.
After normalization,
\begin{equation}\label{eq:Tziotziou-cover-normalized}
N(\sqrt nB_2^n,t\widetilde K)\ls\exp\left(\frac{C\delta_n^2n}{t^2}\right).
\end{equation}
For $2\ls p\ls\sqrt n$, $t>0$ and $s>0$, submultiplicativity, Paouris' estimate and~\eqref{eq:Tziotziou-cover-normalized} give
$$\ln N\left(Z_p(X),t\sqrt{\frac pn}\,\widetilde K\right) \ls \ln N\left(Z_p(X),\frac ts\sqrt p B_2^n\right)    +\ln N(\sqrt nB_2^n,s\widetilde K) \ls Cn\left(\frac{s^2}{t^2}+\frac{\delta_n^2}{s^2}\right).$$
Choosing $s^2=\delta_n t$ yields
\begin{equation}\label{eq:Tziotziou-factorization-result}
\ln N\left(Z_p(X),t\sqrt{\frac pn}\,\widetilde K\right) \ls \frac{C\delta_n n}{t}.
\end{equation}
In particular, taking $t=\sqrt{n/p}$ in~\eqref{eq:Tziotziou-factorization-result},
\begin{equation}\label{eq:Zp-in-K-cover}
\ln N(Z_p(X),\widetilde K)\ls C\sqrt{np\ln(en)}.
\end{equation}
Since $I_1(X,\widetilde K)=\E\|Y\|_K=\frac{n}{n+1}\simeq1$,
Corollary~\ref{cor:full-range-body-factorization} gives the same order.
Thus~\eqref{eq:Zp-in-K-cover} gives, in this special case, a separate derivation of the same order using Tziotziou's covering argument with the updated mean norm input, rather than a stronger endpoint of the general factorization.
\appendix

%%%%%%%%%%%%%%%%%%%%%%%%%%%%%%%%%%%%%%%%%%%%%%%%%%%%%%%%%%%%%%%%%%%%%%%%%%%%%%%%%%%%%%%%%%%%%%%%%%%%%%%%%%%%%%%%%%%%%%%%%%%%%%%%%%%%%%
\section{Additional consequences of the moment comparison}\label{appendix:additional-sudakov}
%%%%%%%%%%%%%%%%%%%%%%%%%%%%%%%%%%%%%%%%%%%%%%%%%%%%%%%%%%%%%%%%%%%%%%%%%%%%%%%%%%%%%%%%%%%%%%%%%%%%%%%%%%%%%%%%%%%%%%%%%%%%%%%%%%%%%%

The following results are consequences of the moment comparison and are not needed in the proof of the two main theorems.
They are included here in their full form.

\begin{corollary}
Let $n\gr2$, let $X$ be isotropic and log-concave in $\R^n$, and let $T\subseteq\R^n$ be finite with $|T|\gr2$.
Put $\operatorname{diam}(T)=\max_{s,t\in T}|t-s|$.
For every $\varepsilon>0$ and every $q\gr1$,
\begin{equation}\label{eq:moment-Sudakov}
\left\|\sup_{s,t\in T}\langle t-s,X\rangle\right\|_q \gr \frac{c}{\sqrt{\ln(en)+q}} \left(\varepsilon\sqrt{\ln N(T,\varepsilon B_2^n)}+\sqrt q\,\operatorname{diam}(T)\right),
\end{equation}
and
\begin{equation}\label{eq:quantile-Sudakov}
\mathbb P\left\{\sup_{s,t\in T}\langle t-s,X\rangle \gr \frac{c}{\sqrt{\ln(en)+q}} \left(\varepsilon\sqrt{\ln N(T,\varepsilon B_2^n)}+\sqrt q\,\operatorname{diam}(T)\right)\right\} \gr e^{-Cq}.
\end{equation}
For every $1\ls q\ls c\ln(en)/\psi(X)^2$,
\begin{equation}\label{eq:moment-Sudakov-upper}
\left\|\sup_{s,t\in T}\langle t-s,X\rangle\right\|_q \ls C\sqrt{\ln(en)} \left(\E\sup_{s,t\in T}\langle t-s,G\rangle +\sqrt q\,\operatorname{diam}(T)\right),
\end{equation}
and
\begin{equation}\label{eq:quantile-Sudakov-upper}
\mathbb P\left\{\sup_{s,t\in T}\langle t-s,X\rangle >C\sqrt{\ln(en)} \left(\E\sup_{s,t\in T}\langle t-s,G\rangle +\sqrt q\,\operatorname{diam}(T)\right)\right\} \ls e^{-q}.
\end{equation}
In particular, the two upper estimates hold uniformly for $1\ls q\ls c\sqrt{\ln(en)}$.
\end{corollary}

\begin{proof}
Let $K=\conv(T-T)$.
Then $K$ is origin symmetric, and $h_K(x)=\sup_{s,t\in T}\langle t-s,x\rangle$ is a seminorm.
Theorem~\ref{th:intro-moment-comparison}(i) gives
$$ \|h_K(X)\|_q\gr\frac{c}{\sqrt{\ln(en)+q}}\|h_K(G)\|_q. $$
Fix $t_0\in T$ and choose $s_0,t_1\in T$ with $|t_1-s_0|=\operatorname{diam}(T)$.
The Gaussian Sudakov minoration and the inclusions $t-t_0,t_1-s_0,s_0-t_1\in K$ give
\begin{align*}
 \|h_K(G)\|_q&\gr\E h_K(G) \gr c\varepsilon\sqrt{\ln N(T,\varepsilon B_2^n)},\\
 \|h_K(G)\|_q&\gr\|\langle t_1-s_0,G\rangle\|_q \gr c\sqrt q\,\operatorname{diam}(T).
\end{align*}
Since $\max\{u,v\}\gr(u+v)/2$, the preceding estimate gives \eqref{eq:moment-Sudakov}.

To obtain the lower quantile estimate, Borell's lemma~\cite{Borell-1974} gives regular growth of moments of seminorms under log-concave measures; in particular,
$$ \|h_K(X)\|_{2q}\ls C_0\|h_K(X)\|_q, \qquad q\gr1. $$
Applying the Paley--Zygmund inequality to $h_K(X)^q$ and using the preceding estimate,
$$ \mathbb P\left\{h_K(X)\gr\frac12\|h_K(X)\|_q\right\} \gr \frac14\left(\frac{\|h_K(X)\|_q}{\|h_K(X)\|_{2q}}\right)^{2q} \gr e^{-Cq}.$$
Together with~\eqref{eq:moment-Sudakov}, this proves~\eqref{eq:quantile-Sudakov}.

For the upper estimates, assume that $1\ls q\ls c\ln(en)/\psi(X)^2$.
Theorem~\ref{th:intro-moment-comparison}(iii) applied to the seminorm $h_K$ gives
$$ \|h_K(X)\|_q\ls C\sqrt{\ln(en)}\,\|h_K(G)\|_q. $$
The Gaussian concentration inequality, or equivalently the standard moment estimate for Lipschitz functions of a Gaussian vector, gives
$$ \|h_K(G)\|_q \ls \E h_K(G)+C\sqrt q\,\operatorname{diam}(T), $$
because the Euclidean Lipschitz constant of $h_K$ is $\max_{u\in K}|u|=\operatorname{diam}(T)$.
This proves~\eqref{eq:moment-Sudakov-upper}.
Finally, Markov's inequality applied to $h_K(X)^q$ at the level $e\|h_K(X)\|_q$ proves~\eqref{eq:quantile-Sudakov-upper} after changing the absolute constant.
\end{proof}

\begin{corollary}
Let $n\gr2$, let $X$ be isotropic and log-concave in $\R^n$, let $p\gr2$, and let $T\subseteq\R^n$ be finite with $|T|\gr e^p$ and
$$ \|\langle t-s,X\rangle\|_p\gr A\qquad(s,t\in T,\ s\ne t). $$
Then, for every $q\gr1$,
\begin{equation}\label{eq:quantile-Lp-Sudakov}
\mathbb P\left\{\sup_{s,t\in T}\langle t-s,X\rangle \gr \frac{cA}{\sqrt{\ln(en)+q}} \left(\frac1{\sqrt p}+\frac{\sqrt q}{p}\right)\right\} \gr e^{-Cq}.
\end{equation}
In particular,
\begin{equation}\label{eq:quantile-Lp-natural-scale}
\mathbb P\left\{\sup_{s,t\in T}\langle t-s,X\rangle \gr \frac{cA}{\sqrt{p\ln(en)}}\right\} \gr \exp\bigl(-C\min\{p,\ln(en)\}\bigr).
\end{equation}
\end{corollary}

\begin{proof}
By~\eqref{eq:one-dimensional-moment-growth} and isotropy, $|t-s|\gr cA/p$ for distinct $s,t\in T$.
Thus, for $\varepsilon=cA/p$ with a sufficiently small absolute constant,
$$ N(T,\varepsilon B_2^n)\gr|T|\gr e^p, \qquad \operatorname{diam}(T)\gr\frac{cA}{p}. $$
\eqref{eq:quantile-Lp-Sudakov} follows from~\eqref{eq:quantile-Sudakov}.
For~\eqref{eq:quantile-Lp-natural-scale}, take $q=\max\{1,\min\{p,\ln(en)\}\}$ and absorb the bounded case into the absolute constants.
\end{proof}

\begin{proposition}\label{prop:large-family}
Let $X$ be isotropic and log-concave in $\R^n$, let $p\gr2$ and let $T\subseteq\R^n$ be finite with $|T|\gr2$.
Put $d=\dim\operatorname{span}(T-T)$.
If
$$ \|\langle t-s,X\rangle\|_p\gr A\qquad(s,t\in T,\ s\ne t), $$
then
\begin{equation}\label{eq:large-family-bound}
\E\sup_{s,t\in T}\langle t-s,X\rangle \gr \frac{cA}{p\sqrt{\ln(ed)}} \max\left\{\sqrt{\ln |T|},\sqrt d\bigl(|T|^{1/d}-1\bigr)\right\}.
\end{equation}
In particular, the left hand side is at least $cA$ provided that
$$ \ln |T|\gr C\min\left\{ p^2\ln(ed), d\ln\left(1+p\sqrt{\frac{\ln(ed)}d}\right) \right\}. $$
\end{proposition}

\begin{proof}
Let $E=\operatorname{span}(T-T)$ and let $Y=P_EX$.
Then $Y$ is isotropic and log-concave in the $d$-dimensional space $E$.
Fix $t_0\in T$ and replace $T$ by $T-t_0$, so that $0\in T\subseteq E$.
Put $K=\conv(T)$.
By \eqref{eq:one-dimensional-moment-growth}, $|t-s|\gr\frac{cA}{p}\qquad(s\ne t)$.
Since $0\in K$, the function $h_K$ is a gauge and $\sup_{s,t\in T}\langle t-s,X\rangle=\sup_{s,t\in T}\langle t-s,Y\rangle\gr h_K(Y)$ pointwise.
The comparison of Bizeul and Klartag in $E$, followed by the Gaussian Sudakov minoration, gives
$$ \E\sup_{s,t\in T}\langle t-s,X\rangle \gr \frac{cA}{p\sqrt{\ln(ed)}}\sqrt{\ln |T|}. $$

For the second estimate, put $r=cA/p$.
The sets $ t+(r/2)B_2^E$, for $t\in T, $ have disjoint interiors and are contained in $K+(r/2)B_2^E$.
Hence $\operatorname{vrad}_E\left(K+\frac r2B_2^E\right)\gr\frac r2|T|^{1/d}$.
By Urysohn's inequality, if $G_E$ is standard Gaussian in $E$, then
$$ \E h_K(G_E)+\frac r2\E|G_E| =\E h_{K+(r/2)B_2^E}(G_E) \gr\frac r2|T|^{1/d}\E|G_E|. $$
Since $\E|G_E|\simeq\sqrt d$, it follows that
$$ \E h_K(G_E)\gr cr\sqrt d\bigl(|T|^{1/d}-1\bigr). $$
Applying the comparison of Bizeul and Klartag in $E$ once more gives
$$ \E\sup_{s,t\in T}\langle t-s,X\rangle \gr \frac{cA}{p\sqrt{\ln(ed)}}\sqrt d\bigl(|T|^{1/d}-1\bigr). $$
Together with the first estimate and $\max\{u,v\}\gr(u+v)/2$, this proves~\eqref{eq:large-family-bound}.
The last assertion follows by solving each of the two sufficient inequalities for $\ln|T|$.
\end{proof}

Lata\l a's estimate~\cite[Lemma~2.8]{Latala-Sudakov-2014} already gives the following lower bound involving the dimension:
$$ \E\sup_{s,t\in T}\langle t-s,X\rangle\gr \frac{cA}{p}\bigl(|T|^{1/d}-1\bigr). $$
The second term in~\eqref{eq:large-family-bound} is obtained from the same volumetric mechanism, followed by Urysohn's inequality and the reverse gauge comparison, and gains the factor $\sqrt{d/\ln(ed)}$.
Lata\l a's theorem~\cite[Theorem~28]{Latala-Problems-2017} gives an absolute lower bound under the assumption $|T|\gr\exp(e^p)$, which does not involve the dimension.
Proposition~\ref{prop:large-family} gives a complementary condition that does involve it.

%%%%%%%%%%%%%%%%%%%%%%%%%%%%%%%%%%%%%%%%%%%%%%%%%%%%%%%%%%%%%%%%%%%%%%%%%%%%%%%%%%%%%%%%%%%%%%%%%%%%%%%%%%%%%%%%%%%%%%%%%%%%%%%%%%%%%%
\section{Comparison with the Mendelson--Milman--Paouris estimates}\label{appendix:MMP-comparison}
%%%%%%%%%%%%%%%%%%%%%%%%%%%%%%%%%%%%%%%%%%%%%%%%%%%%%%%%%%%%%%%%%%%%%%%%%%%%%%%%%%%%%%%%%%%%%%%%%%%%%%%%%%%%%%%%%%%%%%%%%%%%%%%%%%%%%%

For comparison, the weak part of the Mendelson--Milman--Paouris program becomes unconditional after the resolution of the slicing problem.
Applied to the same self generated target, it gives the following dimension dependent estimate.

\begin{proposition}\label{prop:MMP-entropy}
Let $X$ be centered, nondegenerate and log-concave in $\R^n$, let $r\gr1$ and let $u\gr1$.
If $1\ls p\ls n$, then
\begin{equation}\label{eq:MMP-entropy}
\ln\mathsf M\left(Z_p(X),Cu\mathcal I_r(X)Z_r(X)^\circ\right)
\ls C\left[1+n\left(\frac{\ln(e+u\sqrt{n/p})}{u\sqrt{n/p}}\right)^{1/3}\right].
\end{equation}
If $p\gr n$, then
\begin{equation}\label{eq:large-p-natural-entropy}
\ln\mathsf M\left(Z_p(X),Cu\mathcal I_r(X)Z_r(X)^\circ\right) \ls C\left(1+\frac pu\right).
\end{equation}
\end{proposition}

\begin{proof}
Suppose first that $X$ is origin symmetric, and let $\mu$ be its law.
The resolution of the slicing problem implies that every marginal of $\mu$ has isotropic constant bounded by an absolute constant.
In the terminology of Mendelson, Milman and Paouris, this means that $\mu$ is $1$-pure.
Their estimate~\cite[Theorem~9.1]{Mendelson-Milman-Paouris}, applied with $q=1$ and $L=Z_r(X)^\circ$, gives
$$ \ln\mathsf M\left(Z_p(X),t\sqrt{\frac pn}\,m_1(\mu,L)L\right) \ls 2+Cn\left(\frac{\ln(e+t)}t\right)^{1/3}, \qquad 1\ls p\ls n. $$
Since $m_1(\mu,L)\simeq I_1(X,L)=\mathcal I_r(X)$, taking $t=u\sqrt{n/p}$ proves~\eqref{eq:MMP-entropy}.
For $p\gr n$, their large $p$ estimate~\cite[Theorem~4.1]{Mendelson-Milman-Paouris}, again with $q=1$, gives
$$ \ln\mathsf M\left(Z_p(X),Cu\,m_1(\mu,L)L\right) \ls 2+\frac pu, $$
and proves~\eqref{eq:large-p-natural-entropy}.

For a centered vector $X$, let $Y=X-X'$, where $X'$ is an independent copy.
The vector $Y$ is symmetric and log-concave, while $Z_q(X)\subseteq Z_q(Y)\subseteq2Z_q(X)$ and $\mathcal I_q(X)\ls \mathcal I_q(Y)\ls4\mathcal I_q(X)$.
In particular, $\mathcal I_r(Y)Z_r(Y)^\circ\subseteq4\mathcal I_r(X)Z_r(X)^\circ$, because $Z_r(X)\subseteq Z_r(Y)$.
Since also $Z_p(X)\subseteq Z_p(Y)$, the symmetric estimates for $Y$ and monotonicity of packing numbers prove both assertions after changing the absolute constant in the packing scale.
\end{proof}

Theorem~\ref{th:polar-centroid-entropy} and Proposition~\ref{prop:MMP-entropy} are complementary.
In the isotropic range $2\ls r\ls n$, our estimate is dimension free and gives
$$ \ln\mathsf M\left(Z_p(X),Cu\mathcal I_r(X)Z_r(X)^\circ\right)\ls Cp\left(\frac r{u^2}+\frac{\sqrt r}{u}\right). $$
The Mendelson--Milman--Paouris estimate has no explicit dependence on $r$ beyond the normalization of the target, but retains the ambient dimension:
$$ C\left[1+n\left(\frac{\ln(e+u\sqrt{n/p})}{u\sqrt{n/p}}\right)^{1/3}\right] $$
when $p\ls n$, and $C(1+p/u)$ when $p\gr n$.
Neither bound dominates the other for all values of $p,r$ and $u$; the first is suited to dimension free control with explicit dependence on the moment scale $r$, while the second can be preferable at large separation scales.

Combining Theorem~\ref{th:polar-centroid-entropy}, Corollary~\ref{cor:centroid-width-entropy} and Proposition~\ref{prop:MMP-entropy}, we obtain, for $2\ls p,r\ls n$ and $u\gr1$,
\begin{align*}
&\ln\mathsf M\left(Z_p(X),Cu\mathcal I_r(X)Z_r(X)^\circ\right)\\
&\quad\ls C\min\left\{ p\left(\frac r{u^2}+\frac{\sqrt r}{u}\right),\; \frac{\sqrt{np\ln(e+p)\ln(e+r)}}{u},\; 1+n\left(\frac{\ln(e+u\sqrt{n/p})}{u\sqrt{n/p}}\right)^{1/3} \right\}.
\end{align*}
At the scale $u=1$, this gives
$$ \ln\mathsf M\left(Z_p(X),C\mathcal I_r(X)Z_r(X)^\circ\right) \ls C\min\left\{ pr,\;\sqrt{np\ln(e+p)\ln(e+r)},\; n^{5/6}p^{1/6}\ln^{1/3}\left(e+\frac np\right) \right\}. $$

%%%%%%%%%%% End of paper body %%%%%%%%%%%%%%%%%%%%%%%%%%%%%%%
\bigskip

\noindent {\bf Acknowledgements.} 
The first named author acknowledges support by a PhD scholarship from the National Technical University of Athens.
The second named author acknowledges support by the Hellenic Foundation for Research and Innovation (H.F.R.I.) under the ``4th Call for H.F.R.I. research projects to support Postdoctoral Researchers'' (Project Number: 28948).
The authors thank Apostolos Giannopoulos for useful discussions.

\bigskip

%%%%%%%%%%%%%%%%%%%%%%%%%%%%%%%%%%%%%%%%%%%%%%%%%%%%%%%%%%%%%%%%%%%%%%%%%
%%%%%%%%%%%%%%%%%%%%%%%%%%%%%%%%%%%%%%%%%%%%%%%%%%%%%%%%%%%%%%%%%%%%%%%%%

\bigskip

\noindent {\bf Keywords:} log-concave random vectors, Sudakov minoration, weak and strong moments, $L_p$-centroid bodies, metric entropy, KLS parameter.

\smallskip

\noindent {\bf 2020 Mathematics Subject Classification:} Primary 60E15; Secondary 52A23, 46B09, 46B07.

\bigskip

\noindent \textsc{Antonios \ Hmadi}: School of Applied Mathematical and Physical Sciences, National Technical University of Athens, Department of Mathematics, Zografou Campus, GR-157 80, Athens, Greece.

\smallskip

\noindent \textit{E-mail:} \texttt{ahmadi@mail.ntua.gr}

\bigskip 

\medskip 

\noindent \textsc{Dimitrios-Marios \ Liakopoulos}: School of Applied Mathematical and Physical Sciences, National Technical University of Athens, 
Department of Mathematics, Zografou Campus, GR-157 80, Athens, Greece.

\smallskip

\noindent \textit{E-mail:} \texttt{dm\_liakopoulos@mail.ntua.gr}

\end{document}